%% file: main.tex
\documentclass[11pt,reqno]{amsart}

\usepackage{amsmath,amssymb,amsfonts,amsthm}
\usepackage{accents}
\usepackage{breqn}
\usepackage{cite}
\usepackage{enumerate,enumitem}
\usepackage{float}
\setlist[enumerate]{label=(\alph*)}
\usepackage[top=30truemm,bottom=30truemm,left=30truemm,right=30truemm]{geometry}
\usepackage[breaklinks=true,colorlinks=true,linkcolor=blue,citecolor=blue,filecolor=blue,urlcolor=blue]{hyperref}
\hypersetup{linktocpage}
\usepackage{mathrsfs}
\usepackage{mathtools}
\usepackage{todonotes}
\usepackage{xspace}
\usepackage[capitalise]{cleveref}

\newtheorem{thm}{Theorem}[section]
\newtheorem*{thm*}{Theorem}
\newtheorem{cor}[thm]{Corollary}
\newtheorem{prop}[thm]{Proposition}
\newtheorem{lem}[thm]{Lemma}
\newtheorem{conj}[thm]{Conjecture}
\newtheorem*{conj*}{Conjecture}
\newtheorem{quest}[thm]{Question}

\theoremstyle{definition}
\newtheorem{defn}[thm]{Definition}

\theoremstyle{remark}

\newtheorem*{rmk}{Remark}

\newcommand{\CC}{\mathbb C}

\newcommand{\QQ}{\mathbb Q}
\newcommand{\RR}{\mathbb R}

\newcommand{\mc}[1]{\mathcal{#1}}

\DeclareMathOperator{\gal}{Gal}
\DeclareMathOperator*{\res}{Res}
\DeclareMathOperator*{\ord}{ord}

\DeclareMathOperator{\Real}{Re}
\DeclareMathOperator{\Imag}{Im}

\newcommand{\mertensconjecture}{na\"ive Mertens-type conjecture\xspace}

\numberwithin{equation}{section}

\newcommand{\adisc}{D}
\newcommand{\disc}{\Delta}

\title{On a Mertens-type conjecture for number fields}

\author{Daniel Hu}
\address{Department of Mathematics, Princeton University, Princeton NJ 08544-1000, USA}
\email{danielhu@princeton.edu}

\author{Ikuya Kaneko}
\address{Division of Physics, Mathematics and Astronomy, California Institute of Technology, Pasadena, CA 91125, USA}
\email{ikuyak@icloud.com}
\urladdr{\href{https://sites.google.com/view/ikuyakaneko/}{https://sites.google.com/view/ikuyakaneko/}}

\author{Spencer Martin}
\address{UCLA Department of Mathematics, Los Angeles, CA 90024, USA}
\email{stmartin@math.ucla.edu}

\author{Carl Schildkraut}
\address{Department of Mathematics, Massachusetts Institute of Technology, Cambridge, MA 02139-4307, USA}
\email{carlsc@mit.edu}

\subjclass[2020]{Primary: 11N56; Secondary: 11N64}

\keywords{}

\date{\today}

\begin{document}

\begin{abstract}
We introduce a number field analogue of the Mertens conjecture and demonstrate its falsity for all but finitely many number fields of any given degree. We establish the existence of a logarithmic limiting distribution for the analogous Mertens function, expanding upon work of Ng. Finally, we explore properties of the generalized Mertens function of certain dicyclic number fields as consequences of Artin factorization. 
\end{abstract}

\maketitle

\setcounter{tocdepth}{1}
\tableofcontents 

\section{Introduction}\label{sec:intro}
\input{1introduction}

\section{Preliminaries}\label{sec:prelim}

\input{2prelims}

\section{Unconditional failure of the \mertensconjecture}\label{sec:disproof}

\input{3jurkat}

\section{Proof of Theorem \ref{thm:limiting-distribution-abelian}}\label{sec:limdist}

\input{4limdist}

\section{Proofs of analytic results}\label{sec:details}

\input{5details}

\section{Speculations on one-sided growth of Mertens functions}
\label{sec:one-sided-growth}
\input{6one-sided-growth}

\section{Conclusions and conjectures}
\label{sec:conclusion-conjectures}
\input{7conclusions-and-conjectures}


\bibliographystyle{alpha}
\bibliography{bib}

\end{document}

%% file: 1introduction.tex
For a number field $K$, we define the \textit{M\"{o}bius function} $\mu(\mathfrak{a}) = \mu_{K}(\mathfrak{a})$ assigning an integer to each integral ideal $\mathfrak{a}$, according to the rule
\begin{equation*}
    \mu(\mathfrak{p}^k) \coloneqq 
    \begin{cases}
        1 & \text{if $k = 0$} \\
        -1 & \text{if $k = 1$} \\
        0 & \text{if $k \geq 2$}
    \end{cases}
\end{equation*}
for prime ideals $\mathfrak{p}$. This is extended multiplicatively by the unique factorization of ideals. The summatory function of the M\"{o}bius function is the \textit{Mertens function}
\begin{equation*}
    M_K(x) \coloneqq {\sum_{N(\mathfrak{a}) \leq x}} \mu(\mathfrak{a}),
\end{equation*}
where $\mu_K(x) \coloneqq \sum_{N(\mathfrak{a}) = x} \mu(\mathfrak{a})$ is replaced with $\frac{1}{2} \mu_K(x)$ if $x$ is an integer. With this~convention, the function $M_K(x)$ is expressed as the inverse Mellin transform of $1/s\zeta_K(s)$ via the formula
\begin{equation}\label{eq:M_K(x)}
    \frac{1}{\zeta_K(s)} = s \int_1^\infty \frac{M_K(x)}{x^{s+1}} \, dx.
\end{equation}
This relation can be seen as the definition of $M_K(x)$ as a matter of practical convenience.

\subsection{The \mertensconjecture over a number field}
For $K = \mathbb{Q}$, the classical conjecture~of Mertens~\cite{Mertens} in 1897 asserts that
\begin{equation*}
    |M(x)|~\leq \sqrt{x} \quad \text{for all} \quad x \geq 1.
\end{equation*}
Knowing that the Riemann hypothesis is equivalent to the weaker statement that $M(x) = O(x^{1/2 + \varepsilon})$ for all $\varepsilon>0$~\cite[Theorem~14.25(c)]{Titchmarsh}, this seemed to represent a viable avenue towards the Riemann hypothesis for the Riemann zeta function $\zeta(s)$. Before the landmark work of Ingham~\cite{Ingham} in 1942, preliminary calculations~of~Mertens and von Sterneck even compelled the hypothesis $x^{-1/2}|M(x)|~\leq \frac{1}{2}$ for sufficiently large $x$. Nevertheless, using lattice basis reduction algorithms, the conjecture was disproven by Odlyzko and te Riele~\cite{OdlyzkoteRiele} in 1985, who obtained explicit bounds larger in absolute value than $1$ for $x^{-1/2}M(x)$ on either side in the limit. The current record in this direction has been achieved by Hurst~\cite{Hurst} in 2018, namely
\begin{equation*}
    \liminf_{x \rightarrow \infty} \frac{M(x)}{x^{1/2}} < -1.837625 \text{ and }\limsup_{x \rightarrow \infty} \frac{M(x)}{x^{1/2}} > 1.826054.
\end{equation*}
It is now common belief that $x^{-1/2}M(x)$ grows arbitrarily large in both directions, but this has not yet been proven unconditionally~\cite{OdlyzkoteRiele}. Indeed, Ingham~\cite[Theorem~A]{Ingham} showed that,~assuming the Riemann hypothesis and that the imaginary parts of the nontrivial zeros of $\zeta(s)$ are $\mathbb{Q}$-linearly independent, the claim is true. This work marked the first serious doubt regarding the Mertens conjecture until its subsequent disproof. In fact, preliminary computational support for the linear independence hypothesis is supplied in the proof of Best and Trudgian~\cite{BT}, which builds upon a previous result of Bateman et al.~\cite{Bateman}. 

Several mathematicians have investigated analogues of this conjecture in other settings, some achieving corresponding disproofs. Anderson~\cite{Anderson} studied the case of cusp forms of sufficiently large weight $k \equiv 2 \pmod 4$ on the full modular group $\mathrm{SL}_{2}(\mathbb{Z})$. Grupp~\cite{Grupp} generalized the~work of Anderson to cusp forms of any~even weight $k$. Humphries~\cite{Humphries14} considered Mertens-type conjectures for function fields of smooth projective curves over $\mathbb{F}_q$. Our work focuses on the~case of number fields and seeks a disproof of the relevant conjecture in emulation of previous work.

Following the aforementioned works, it is most relevant to formulate a conjecture based~on the limiting behavior of the arithmetic function $M_K(x)$. Let $K$ be a number field, and write
$$M_K^-=\liminf_{x\to\infty}\frac{M_K(x)}{x^{1/2}}, \qquad M_K^+=\limsup_{x\to\infty}\frac{M_K(x)}{x^{1/2}}.$$
We now state the guiding question of our work.

\begin{conj*}[The \mertensconjecture over $K$] $-1\leq M_K^-\leq M_K^+\leq 1$.
\end{conj*}

Although this is in some ways a na\"{i}ve generalization of the original Mertens conjecture, we have chosen $1$ as the critical constant following tradition. Yet, it can be seen that this choice will be of strategic importance for our method.

A corresponding conclusion to that of Ingham's theorem can be shown to hold for the function $M_K(x)$ (see~\cref{conj:li},~\cref{thm:fe}, and~\cref{lem:large-horiz-strip}), which leads to the prediction that the \mertensconjecture over $K$ is indeed false for all $K$. Towards partial disproofs of~the conjecture, the quadratic fields $K$ are the most straightforward to consider and we give certain unconditional results.

\begin{thm}\label{thm:mertens-quadfields}
We have the following statements:
\begin{enumerate}[label=(\alph*), font=\normalfont]
    \item Let $K$ be an imaginary quadratic extension of $\mathbb{Q}$, with the exception of $K = \mathbb{Q}(\sqrt{-3})$. Then the \mertensconjecture over~$K$ is false. 
    \item Let $K$ be a real quadratic extension of $\mathbb{Q}$, with the exception of $K = \mathbb{Q}(\sqrt{5})$. Then the \mertensconjecture over $K$ is false. 
\end{enumerate}
\end{thm}

\begin{rmk}\label{rmk:q-sqrt-neg-3}
While the basic statement $|M_K(x)|~\leq x^{1/2}$ for all $x \geq 1$ is false for $K = \mathbb{Q}(\sqrt{-3})$ and $K = \mathbb{Q}(\sqrt{5})$ (an easy computation demonstrates that $M_{\mathbb{Q}(\sqrt{-3})}(7) = -3$ and $M_{\mathbb{Q}(\sqrt{5})}(11) = -4$), the methods of this paper are not immediately capable of demonstrating that the \mertensconjecture fails for these fields (as this requires knowledge of limiting behavior), unlike the other cases of the previous theorem. It is, however, likely that computations analogous to those~used~in \cite{OdlyzkoteRiele} to falsify the original Mertens conjecture would suffice for both of these quadratic fields as well. 
\end{rmk}

The techniques of the proof of Theorem~\ref{thm:mertens-quadfields} allow us to arrive at a similar result for general number fields of degree $n_K > 2$. In what follows, the \textit{signature} $(r_1,r_2)$ of a number field $K$ is the ordered pair of non-negative integers that encodes the number $r_1$ of real embeddings and $r_2$ of complex conjugate pairs of embeddings of $K$.

\begin{thm}\label{thm:mertens-genfields}
Fix a signature $(r_1,r_2)$. There exists some $\mathcal{D} = \mathcal{D}_{r_1,2r_2} > 0$ depending only on $(r_1,r_2)$ for which the \mertensconjecture is false~for every extension $K$ of $\mathbb{Q}$ of signature $(r_1,r_2)$ with absolute discriminant $\adisc_K > \mathcal{D}$. 
\end{thm}

We establish these results in \cref{sec:disproof}.

\subsection{Logarithmic limiting distributions}
In order to study the Mertens function over a number field, we now define another object which serves as a chief interest of our work.

\begin{defn}\label{def:limdist}
We say that a function $\vec{\phi}: [0,\infty) \rightarrow \mathbb{R}^{\ell}$ possesses a \textit{limiting distribution} $\nu$ on $\mathbb{R}^{\ell}$ if $\nu$ is a probability measure on $\mathbb{R}^{\ell}$ and
\begin{equation*}
    \lim_{Y \rightarrow \infty} \frac{1}{Y} \int_{0}^Y f(\vec{\phi}(y)) \, dy = \int_{\mathbb{R}^{\ell}} f(x) \, d\nu(x)
\end{equation*}
for all bounded continuous real-valued functions $f$ on $\mathbb{R}^{\ell}$.
\end{defn}

Recent years have seen great refinements in the probabilistic methods used to study number-theoretic functions. The influential work of Rubinstein and Sarnak \cite{RS} in establishing the existence of limiting distributions pertaining to various questions on R\'{e}nyi--Shanks prime number races, with its extensive generalizations in the work of Ng~\cite{Ng}, Humphries~\cite{Humphries13},~and Akbary--Ng--Shahabi~\cite{ANS}, has seen wide applications to the summatory functions in number theory, including those of the Liouville function and the M\"{o}bius function. Pertaining to~the classical Mertens function $M(x)$, Ng~\cite{Ng} established that the function
\begin{equation*}
    \phi(y) \coloneqq e^{-y/2}M(e^y)
\end{equation*}
possesses a limiting distribution $\nu$ on $\mathbb{R}$, assuming the Riemann hypothesis for and the following conjecture of Gonek and Hejhal on the discrete moments of $\zeta'(s)$ at the zeros:
\begin{equation}\label{eq:gonek-hejhal}
    J_{-1}(T) \ll T, 
\end{equation}
where, with the convention $\rho = \frac{1}{2} + i\gamma$,
\begin{equation*}
    J_k(T) \coloneqq \sum_{\substack{0 < \gamma \leq T \\ \zeta(\frac{1}{2} + i\gamma) = 0}} |\zeta'(\rho)|^{2k},
\end{equation*}

We prove the following generalization of Ng's result.

\begin{thm}
\label{thm:limiting-distribution-abelian}
Let $K$ be an abelian number field. Assume the Riemann hypothesis for~$\zeta_K(s)$, and the following extension of~\eqref{eq:gonek-hejhal}:
\begin{equation}\label{eq:gonek-hejhal-nfield}
    J_{-1}^K(T) \coloneqq \sum_{\substack{0 \leq \gamma \leq T \\ \zeta_K(\frac{1}{2} + i\gamma) = 0}} \frac{1}{|\zeta_K'(\rho)|^2} \ll_{\alpha} T^{1 + \alpha}.
\end{equation}
for some $0 \leq \alpha < 2 - \sqrt{3}$. Then the function
\begin{equation*}
    \phi_K(y) \coloneqq e^{-y/2}M_K(e^y)
\end{equation*}
possesses a limiting distribution $\nu_K$ on $\mathbb{R}$:
\begin{equation*}
    \lim_{Y \rightarrow \infty} \frac{1}{Y} \int_0^Y f(\phi_K(y)) \, dy = \int_{-\infty}^\infty f(x) \, d\nu_K(x)
\end{equation*}
for all bounded continuous real-valued functions $f$ on $\mathbb{R}$.
\end{thm}

This is proven in \cref{sec:limdist}.

To be well-defined, the bound~\eqref{eq:gonek-hejhal-nfield} implies that all the nontrivial zeros of $\zeta_K(s)$ are simple. This dictates, in consequence, the non-vanishing at $s = \frac{1}{2}$ of $\zeta_K(s)$, the sign of whose functional equation always prescribes at $s = \frac{1}{2}$ a zero of even multiplicity. 

The validity of these suppositions is most transparent when $K/\mathbb{Q}$ is a Galois extension, due to Artin factorization for the Dedekind zeta functions of such number fields, and is markedly different between abelian and non-abelian Galois extensions. Conjecturally, it emerges in~the abelian case from the conventional hypotheses that no two Dirichlet $L$-functions share a nontrivial zero and that no Dirichlet $L$-function vanishes at the central point $s = \frac{1}{2}$. Both of these assertions are weaker consequences of the Grand Simplicity Hypothesis for the Dirichlet $L$-functions (cf.~\cite{RS}). The latter conjecture is remarkably well-studied, along with the work of Balasubramanian--Murty~\cite{BM} and Iwaniec--Sarnak~\cite{IS} showing that $L(\frac{1}{2},\chi_q) \neq 0$ for a positive proportion of the primitive characters $\chi_q$ of any sufficiently large modulus $q$.

On the other hand, for non-abelian number fields, the simplicity hypothesis is known to be false unconditionally, with the zeta function of any such number field having infinitely many nontrivial zeros of multiplicity at least $2$ (cf.~\cite{hkms}). By example, the non-vanishing at $s = \frac{1}{2}$ has also been disproven, with the earliest examples of number fields with $\zeta_K(\frac{1}{2}) = 0$ known to Armitage and Serre (see~\cref{sec:one-sided-growth} for further discussion). This renders an obstruction to any extension of Theorem~\ref{thm:limiting-distribution-abelian} to the non-abelian case.

The conjecture \eqref{eq:gonek-hejhal} is the special case $k = -1$ of a conjecture $J_{k}(T) \asymp T(\log T)^{(k+1)^2}$~that was first proposed independently by Gonek~\cite{Gonek1} and Hejhal~\cite{Hejhal} for all $k \in \RR$ in view of their examinations of the discrete moments $J_{k}(T)$. Gonek also predicted the asymptotic formula $J_{-1}(T) \sim \frac{3}{\pi^3}T$ and proved the bound $J_{-1}(T) \gg T$. It was subsequently~proposed~by Keating--Snaith~\cite{KS} that random matrix theoretic heuristics could be used in connection with conjectures about moments of the zeta function. In particular, they modeled the value distribution of $\zeta(s)$ near the critical line by the characteristic polynomial of a large unitary random matrix. Expanding upon this work, Hughes--Keating--O'Connell~\cite{HKO} conjectured asymptotic formulae for all $J_{k}(T)$, $k > -\frac{3}{2}$, giving explicit numerical constants depending on $k$. The random matrix approach has proven to produce reliable conjectures; for instance, the conjecture of Hughes--Keating--O'Connell in the case $k = -1$ agrees with that of Gonek. It is expected that similar heuristics would yield accurate conjectures for positive moments of Dedekind zeta functions as well (see~\cite{GHK,BGM,Heap}). They would also be useful in support of \eqref{eq:gonek-hejhal-nfield}. 

We argue in support of the conjecture in the equation \eqref{eq:gonek-hejhal-nfield} in \cref{sec:limdist}. For now, we have not appealed to random matrix theory, instead relying on prior work and a series of generous but not unreasonable assumptions, by analogy with the case of the Riemann zeta function.

As a direct consequence of the arguments to be presented in \cref{sec:limdist}, it behooves us to mention that Corollary~1.15 of Akbary--Ng--Shahabi~\cite{ANS} yields the following analogue of the weak Mertens conjecture as stated in  \cite[Theorem~1.3]{Ng}.

\begin{thm}\label{thm:weakmertens}
With the same assumptions as in Theorem \ref{thm:limiting-distribution-abelian}, we have
\begin{equation*}
    \int_0^Y \left( \frac{M_K(e^y)}{e^{y/2}}\right)^2 \, dy \sim \beta Y,
\end{equation*}
where $\beta = 2\sum_{\gamma > 0} |\rho \zeta_K'(\rho)|^{-2}$. The assumption $J_{-1}^K(T) \ll_{\alpha} T^{1+\alpha}$ implies that the series defining $\beta$ is convergent.
\end{thm}

In the spirit of the work of Rubinstein--Sarnak and Ng, the existence of a limiting distribution for $\phi_K(y)$ yields a number of consequences which are conditional on the following supplemental analogue of the linear independence for $\zeta(s)$.

\begin{conj}[Linear independence conjecture for $\zeta_K(s)$]\label{conj:li}
The multiset of the non-negative imaginary parts of the nontrivial zeros of $\zeta_K(s)$ is linearly independent over the rationals.
\end{conj}

We remark that this conjecture encompasses the simplicity of the zeros of $\zeta_K(s)$, as well as the non-vanishing of $\zeta_K(s)$ at $s = \frac{1}{2}$, and is hence a viable conjecture for a normal extension $K/\mathbb{Q}$ only when the extension is abelian. Its principal use in the original work of Rubinstein--Sarnak is to obtain an expression for the Fourier transform of the limiting distribution. By~\cite[Corollary~1.3]{ANS}, the following is a direct consequence of Theorem \ref{thm:limiting-distribution-abelian}.

\begin{thm}\label{thm:limdist}
Let $K$ be an abelian extension of $\mathbb{Q}$. Assume the Riemann hypothesis for~$\zeta_K(s)$, that $\zeta_K(\frac{1}{2}) \neq 0$, and that $J_{-1}^K(T) \ll T^{1+\alpha}$ for some $0 \leq \alpha < 2 - \sqrt{3}$, and let $\nu_K$ be the limiting distribution associated to $\phi_K(y)$ as in \cref{thm:limiting-distribution-abelian}. Assume moreover \cref{conj:li}. Then the Fourier transform
\begin{equation*}
    \widehat{\nu}_K(\xi) = \int_{\mathbb{R}} e^{-i x \xi} \, d\nu_K(x)
\end{equation*}
of $\nu_K$ at $\xi \in \mathbb{R}$ exists and is equal to
\begin{equation*}
    \widehat{\nu}_K(\xi) = \prod_{\substack{|\gamma| > 0 \\ \zeta_K(\rho) = 0}} \tilde J_0 \left( \frac{2 \xi}{|\rho \zeta_K'(\rho)|} \right),
\end{equation*}
where $\tilde J_0(z)$ is the Bessel function
\begin{equation*}
    \tilde J_0(z) = \int_0^1 e^{-iz \cos (2\pi t)} \, dt.
\end{equation*}
\end{thm}

This result can be used in pursuit of logarithmic density results: for instance, that the set of numbers $x \geq 1$ in the set $P_{\beta} = \{x^{-1/2}|M_K(x)|~\leq \beta\}$ has a logarithmic density, for certain $\beta > 0$ of interest.

\begin{defn}
For $P \subset [0, \infty)$, set 
\begin{equation*}
    \delta(P) = \lim_{X \rightarrow \infty} \frac{1}{\log X} \int_{t \in P \cap [1.X]} \frac{dt}{t}
\end{equation*}
If the limit exists, we say that the \textit{logarithmic density} of $P$ is $\delta(P)$.
\end{defn}

Following the arguments of 
\cite[Corollary~6.3,~Lemma~6.4]{Humphries13}, we deduce that the Fourier transform $\widehat{\nu}_K$ so constructed is symmetric and observes rapid decay as a function of $\xi$. Hence, $\widehat{\nu}_K$ is absolutely continuous with respect to the Lebesgue measure on $\mathbb{R}$. After a logarithmic change of coordinates, this yields the following extension of the conclusion of \cref{thm:limiting-distribution-abelian} to characteristic functions of well-behaved sets.

\begin{cor}
With the same assumptions as in Theorem \ref{thm:limdist}, 
\begin{equation*}
    \lim_{X \rightarrow \infty} \frac{1}{\log X} \int_{x \in B \cap [1,X]} \, \frac{dx}{x} = \nu_K(B)
\end{equation*}
for all Borel sets $B \subset \mathbb{R}$ with boundary of Lebesgue measure zero.
\end{cor}

Thus, the set $P_{\beta} = \{x \geq 1 \mid |M_K(x)|~\leq \beta \sqrt{x}\}$ has a logarithmic density, under the~assumptions of RH, linear independence, and $J_{-1}^K(T) \ll T^{1+\alpha}$. See \cref{sec:conclusion-conjectures} for further discussion.

\subsection*{Acknowledgements}
We are deeply grateful to Peter Humphries for supervising this project and to Ken Ono for his valuable suggestions. We would also like to thank Winston Heap, David Lowry Duda, and Micah Milinovich for helpful discussions. We are grateful for the generous support of the National Science Foundation (Grants DMS 2002265 and DMS 205118),
National Security Agency (Grant H98230-21-1-0059), the Thomas Jefferson Fund at the University of Virginia, and the Templeton~World Charity Foundation. This research was conducted as part of the 2021 Research Experiences~for Undergraduates at the University of Virginia.

%% file: 2prelims.tex
\subsection{Notation and conventions}
Throughout this article, we use the following conventions:

\begin{itemize}
    \item $K$ is a number field.
    \item $n_K = [K : \QQ]$.
    \item $r_1$ and $2r_2$ are the number of real and complex embeddings of $K$, respectively.
    \item $\disc_K$ is the discriminant of $K$.
    \item $\adisc_K = \vert \disc_K \vert$ is the absolute discriminant of $K$.
    \item $\zeta_K(s)$ is the Dedekind zeta function of $K$.
\end{itemize}

Some theorems in this article apply only to real or imaginary quadratic number fields. In these cases, we will specify any additional hypotheses on $K$. Otherwise, it is assumed that $K$ is a general number field.

The \emph{Riemann hypothesis for $\zeta_K(s)$} will denote the conjecture that all nontrivial zeros of $\zeta_K(s)$ lie on the critical line $\Real(s) = \frac{1}{2}$. In light of this, $\rho$ will always denote a nontrivial zero of~$\zeta_K(s)$ with imaginary part $\gamma$.

\subsection{Analytic properties of Dedekind zeta functions}
In this section, we provide some preliminary statements about Dedekind zeta functions that will be used in the proofs of each of our results. 
We start with the functional equation for $\zeta_K(s)$.

\begin{thm}\label{thm:fe} For any number field $K$, the Dedekind zeta function $\zeta_K(s)$ satisfies
\begin{equation*}
     \zeta_K(1-s)= \zeta_K(s) \left(\frac{\adisc_K}{\pi^{n_K}2^{n_K}}\right)^{s-\frac{1}{2}} \left( \frac{\pi}{2} \right)^{\frac{r_1}{2}}\frac{\Gamma(s)^{r_2}}{(\sin \frac{\pi s}{2})^{r_1} \Gamma(1-s)^{r_1+r_2}}.
\end{equation*}
This functional equation extends $\zeta_K$ to a meromorphic function on $\CC$, which is analytic everywhere except for a simple pole at $s=1$. 
\end{thm}
The Dedekind zeta function has trivial zeros of order $r_2$ at each negative odd integer and of order $r_1+r_2$ at each negative even integer, as well as a trivial zero of order $r_1+r_2-1$ at $s=0$. It also possesses nontrivial zeros, each of which lies in the critical strip $0 < \Real(s) < 1$. By the above functional equation and the reflection principle, these zeros are symmetric about $\Real(s) = \frac{1}{2}$ and the real line.

Next, we give a suitable upper bound, due to Chandrasekharan--Narasimhan, on the Dirichlet series coefficients for $\zeta_K(s)$.

\begin{lem}[Chandrasekharan--Narasimhan~\protect{\cite[Lemma~9]{Chandrasekharan}}]\label{lem:dedekind-coeff-bound}
Let $K$ be any number field, and write in $\Real(s) > 1$,
\begin{equation*}
    \zeta_K(s) = \sum_{n=1}^\infty \frac{a_n}{n^s}.
\end{equation*}
Then $a_n$ is bounded by the coefficient $b_n$ of $n^{-s}$ in $\zeta(s)^{n_K}$ and there exists a constant $C$ depending only on $n_K$ such that
\begin{equation*}
    a_n \leq b_n \ll_{n_K} n^{\frac{C}{\log \log n}}\ \text{ as } n \rightarrow \infty.
\end{equation*}
The same holds for the Dirichlet series coefficients $a_n'$ of $1/\zeta_K(s)$.
\end{lem}

Useful for our purposes will be an estimate for the number of nontrivial zeros of $\zeta_K(s)$ in unit intervals in the critical strip. Define
\begin{equation*}
    N_K(T) \coloneqq \#\{\rho = \beta + i\gamma \mid \zeta_K(\rho) = 0,~ 0 < \beta < 1,~0 \leq |\gamma|~\leq T\}.
\end{equation*}
It is well-known that $N(T+1) - N(T) \ll \log T$ in the case of $K = \mathbb{Q}$. An analogous result holds for any number field $K$.

\begin{lem}\label{cor:zero-count} For any $T\geq 0$, we have that $N_K(T+1)-N_K(T)\ll \log D_K+n_K\log T$.
\end{lem}
\begin{proof}
See for instance~\cite{KN} for suitable estimates of the quantity $N_K(T)$.
\end{proof}

In the explicit formulae to follow in \cref{sec:limdist}, we will also require the following result which gives good upper bounds for $1/\zeta_K(s)$ on certain horizontal lines near the critical strip. We note that it is conditional on the Riemann hypothesis for $\zeta_K(s)$,

\begin{lem}\label{lem:large-horiz-strip}
Assume that $\zeta_K(s)$ satisfies the Riemann hypothesis. Then there exists a constant $C > 0$ such that, for each positive integer $n\geq 4$, there exists some $n\leq T_n<n+1$ such that
\begin{equation*}
    \left|\zeta_K(\sigma+iT_n)\right|\geq \exp\left(-\frac{C\log n}{\log\log n}\right)
\end{equation*}
for $-1\leq \sigma\leq 2$. We denote by $\mathcal{T} = \{T_n\}_{n=4}^\infty$ the sequence so constructed.
\end{lem}
This result is well-known in the case $K = \mathbb{Q}$ (see \cite[Theorem~13.22]{MV} or \cite[Theorem~14.16]{Titchmarsh}) and the proof for general $K$ is mostly analogous. A complete proof can be found in \cref{sec:details}.

\subsection{$\Omega$-type lemmas for $M_K(x)$ and disproofs of the \mertensconjecture in degenerate cases}
Fix any number field $K$. In this section, we provide~various conditional $\Omega$-type theorems for $M_K(x)$ which will reduce the unconditional disproofs in \cref{sec:disproof} to a number of conditional assumptions. The main result of this section is \cref{thm:mertens-very-fail}, which provides three cases in which the \mertensconjecture is guaranteed to fail. Essential in the following proofs will be the following well-known method of Landau.

\begin{lem}[Landau \protect{\cite[Theorem~15.1]{MV}}]\label{lem:Landau}
Suppose that $A(x)$ is a bounded Riemann-integrable function in any finite interval $1 \leq x \leq X$, and that $A(x) \geq 0$ for all $x > X_0$. Let $\sigma_c$ denote the infimum of those $\sigma$ for which $\int_{X_0}^\infty A(x)x^{-\sigma}\, dx < \infty$. Then the function
\begin{equation*}
    F(s) = \int_1^\infty A(x)x^{-s}\, dx
\end{equation*}
is analytic in the half-plane $\Real(s) > \sigma_c$, but not at the point $s = \sigma_c$.
\end{lem}

We shall demonstrate that most cases in which zeros of $\zeta_K(s)$ behave contrary to conventional expectations (i.e. the Riemann hypothesis and simplicity of zeros) will degenerate into falsity of the \mertensconjecture over $K$. One additional possible obstruction present in the case of Dedekind zeta functions, which is not so in the case of the Riemann zeta function, is the presence of nontrivial zeros on the real line. Let
\begin{align*}
    &\Theta = \sup_{\zeta_K(\rho) = 0} \{ \Real(\rho)\} = \max\{\Theta', \Theta''\}\quad \\ &\text{ where }  \Theta' = \sup_{\substack{\zeta_K(\rho) = 0 \\ \rho \in (0,1)}} \rho, \quad \Theta'' = \sup_{\substack{\zeta_K(\rho) = 0 \\ \Imag(\rho) \neq 0}} \{ \Real(\rho)\}.
\end{align*}
If no zero $\rho \in (0,1)$ exists, we set $\Theta' = 0$. In particular, $\Theta \geq \frac{1}{2}$ and the Riemann hypothesis for $\zeta_K(s)$ is the statement $\Theta = \frac{1}{2}$. The following is our first degenerate case.

\begin{prop}\label{prop:noRH-rightmost-notreal}
If $\zeta_K(s)$ does not satisfy the Riemann hypothesis, and $\Theta' < \Theta$, then the \mertensconjecture over $K$ is false. 
More precisely, $M_K(x) = \Omega_{\pm}(x^{\Theta - \epsilon})$ for all $0 < \epsilon < \Theta - \Theta'$.
\end{prop}
\begin{proof}
Suppose that $M(x) < x^{\Theta - \epsilon}$ for all $x > X_0(\epsilon)$. Consider, in view of~\eqref{eq:M_K(x)}, the function
\begin{equation*}
    \frac{1}{s - \Theta + \epsilon} - \frac{1}{s\zeta_K(s)} = \int_1^\infty (x^{\Theta - \epsilon} - M_K(x))x^{-s-1} \, dx.
\end{equation*}
Here the left-hand side has a pole at $\Theta - \epsilon$, but is analytic for real $s > \Theta - \epsilon$, since $\zeta_K(s)$ has no zeros to the right of $\Theta - \epsilon$ on the real line. The integrand of the right-hand side~is~non-negative for all $x > X_0(\epsilon)$. By an application of \cref{lem:Landau}, the above identity holds for $\Real(s) > \Theta - \epsilon$, and both sides are analytic in this half-plane. But by definition of $\Theta$, the function $1/\zeta_K$ has poles with real part $> \Theta - \epsilon$, which is a contradiction. Hence, we deduce that $M(x) = \Omega_+(x^{\Theta - \epsilon})$.

To obtain the $\Omega_-$ estimate, we argue similarly using the identity
\begin{equation*}
    \frac{1}{s - \Theta + \epsilon} + \frac{1}{s\zeta_K(s)} = \int_1^\infty (x^{\Theta - \epsilon} + M_K(x))x^{-s-1} \, dx.
\end{equation*}
We then conclude that $M_K(x) = \Omega_{\pm}(x^{\Theta - \epsilon})$. Specifically, this gives the falsity of the \mertensconjecture over $K$, since $\Theta > \frac{1}{2}$.
\end{proof}

If $\Theta' = \Theta$, then since real zeros must be isolated, there is a zero at $s = \Theta$ which prohibits the application of \cref{lem:Landau} in the proof above. In this case, we give a result in lieu of the previous proposition, which involves the elimination of the polar behavior of $1/s\zeta_K(s)$ at the real nontrivial zeros of $\zeta_K(s)$. More precisely, there exist constants $c_{k,\alpha} \in \mathbb{C}$, $1 \leq k \leq m_{\alpha}$, for each of the finitely many real zeros $\alpha \in (0,1)$, such that the function
\begin{equation}\label{eqn:c-tilde-alpha}
    \frac{1}{s\zeta_K(s)} - \sum_{\substack{\zeta_K(\alpha) = 0 \\ \alpha \in (0,1)}} \sum_{k=1}^{m_\alpha} \frac{c_{k,\alpha}}{(s - \alpha)^k} 
\end{equation}
is analytic on the segment $(0,1)$ on the real line, where $m_\alpha$ is the multiplicity of the zero at $\alpha$. Upon this simplification, we may provide the following result.

\begin{prop}\label{prop:noRH-rightmost-real}
Suppose that $\zeta_K(s)$ does not satisfy the Riemann hypothesis, and $\Theta' = \Theta$, $\frac{1}{2} < \Theta'' \leq \Theta$. Let $c_{k,\alpha}$ be as in the equation \eqref{eqn:c-tilde-alpha}. Then for all $0 < \epsilon < \Theta'' - \frac{1}{2}$,
\begin{equation*}
    M_K(x) - \tilde M_K(x) = \Omega_{\pm}(x^{\Theta'' - \epsilon})
\end{equation*}
where
\begin{equation}\label{eq:tilde-M-K}
    \tilde M_K(x) \coloneqq \sum_{\substack{\zeta_K(\alpha) = 0 \\ \alpha \in (0,1)}} \sum_{k=1}^{m_\alpha} \frac{c_{k,\alpha} x^\alpha (\log x)^{k-1}}{(k-1)!\!}.
\end{equation}
In particular, the \mertensconjecture over $K$ is false.
\end{prop}
\begin{proof}
Suppose that $M_K(x) - \tilde M_K(x) < x^{\Theta - \epsilon}$ for all $x > X_0(\epsilon)$. Now
\begin{align*}
    \frac{1}{s-\Theta +\epsilon} - \frac{1}{s\zeta_K(s)} + \sum_{\substack{\zeta_K(\alpha) = 0 \\ \alpha \in (0,1)}} \sum_{k=1}^{m_\alpha} \frac{c_{k,\alpha}}{(s - \alpha)^k}
    = \int_1^\infty ( x^{\Theta - \epsilon} - M_K(x) + \tilde M_K(x))x^{-s-1}\, dx
\end{align*}
holds for all real $s > \Theta - \epsilon$, and by \cref{lem:Landau} extends to $\Real(s) > \Theta - \epsilon$. From here, the~proof is similar to that of Proposition \ref{prop:noRH-rightmost-notreal}. For the last statement, note that $s = \Theta$ is one of the $\alpha$ in the sum $\tilde M_K(x)$. Since $\Theta > \frac{1}{2}$, this sum surpasses the order of $x^{1/2}$, and since $M_K(x) - \tilde M_K(x)$ is oscillatory we deduce the falsity of the \mertensconjecture over $K$.
\end{proof}

It remains to address the case where $\frac{1}{2} = \Theta'' < \Theta' = \Theta$. We give the following auxiliary result, which is independent of RH.

\begin{prop}\label{prop:zero-on-rightmost-line}
Suppose $\Theta$, $\Theta'$, $\Theta''$ are as above, and that there is a non-real zero $\rho$ of $\zeta_K$ of multiplicity $m \geq 1$ with $\Real(\rho) = \Theta''$, say $\rho = \Theta'' + i\gamma$. Then
\begin{equation*}
    M_K(x) - \tilde M_K(x) = \Omega_{\pm}(x^{\Theta''}(\log x)^{m-1}),
\end{equation*}
where $\tilde M_K(x)$ is defined as in \eqref{eq:tilde-M-K}. More precisely,
\begin{equation*}
    \liminf_{x \rightarrow \infty} \frac{M_K(x) - \tilde M_K(x)}{x^{\Theta''}(\log x)^{m-1}} \leq -\frac{m}{|\rho \zeta_K^{(m)}(\rho)|}<\frac{m}{|\rho \zeta_K^{(m)}(\rho)|}\leq \limsup_{x \rightarrow \infty} \frac{M_K(x) - \tilde M_K(x)}{x^{\Theta''}(\log x)^{m-1}}.
\end{equation*}
\end{prop}
\begin{proof}
Suppose that $M_K(x) - \tilde M_K(x) \leq cx^{\Theta''}(\log x)^{m-1}$ for all $x > X_0$. It suffices to prove that $c \geq m/|\rho \zeta_K^{(m)}(\rho)|$. Consider the function
\begin{dmath*}
    \frac{c(m-1)!\!}{(s-\Theta'')^m} - \frac{1}{s \zeta_K(s)} + \sum_{\substack{\zeta_K(\alpha) = 0 \\ \alpha \in (0,1)}} \frac{c_{k,\alpha} x^{\alpha} (\log x)^{k-1}}{(k-1)!\!}= \int_1^\infty (cx^{\Theta''}(\log x)^{m-1} - M_K(x) + \tilde M_K(x))x^{-s-1} \, dx
\end{dmath*}
for real $s > \Theta''$. By \cref{lem:Landau}, this can be extended to $\Real(s) > \Theta''$. Denote the function $F(s)$. Then
\begin{dmath*}
    F(s) + \frac{1}{2}e^{i\phi}F(s + i\gamma) + \frac{1}{2}e^{-i\phi}F(s - i \gamma) = {\int_1^\infty (cx^{\Theta''}(\log x)^{m-1} - M_K(x) + \tilde M_K(x))(1 + \cos(\phi - \gamma \log x) x^{-s-1})\, dx}
\end{dmath*}
for $\Real(s) > \Theta''$. On the right-hand side, the integral from $1$ to $X_0$ is uniformly bounded,~while the integral from $X_0$ to $\infty$ is non-negative. Thus the $\liminf$ of the right-hand side is bounded below as $s \rightarrow \Theta^+$. As a result, the coefficient of $(s - \Theta'')^{-m}$ in its Laurent series must be non-negative. On the other hand, the left-hand side has a pole of multiplicity $m$ at $s = \Theta''$, at which the Laurent series expansion contains a term $(s - \Theta'')^{-m}$ with coefficient equal to 
\begin{equation*}
    c(m-1)!\! - \frac{m!\! e^{i\phi}}{2\rho \zeta_K^{(m)}(\rho)} - \frac{m!\! e^{-i\phi}}{2\bar \rho \zeta_K^{(m)}(\bar \rho)}.
\end{equation*}
Choosing $\phi$ so that
\begin{equation*}
    e^{i\phi} = \frac{\rho \zeta_K^{(m)}(\rho)}{|\rho \zeta_K^{(m)}(\rho)|}.
\end{equation*}
Then the above is $(m-1)!\!(c - m/|\rho \zeta_K'(\rho)|)$. This quantity must be non-negative, otherwise~the left-hand side would tend to $-\infty$ as $s \rightarrow \Theta^+$. Hence $c \geq m/|\rho \zeta_K^{(m)}(\rho)|$. The $\Omega_-$ case~is~similar.
\end{proof}

In the scenario where $\frac{1}{2} = \Theta'' < \Theta' = \Theta$, we deduce from the previous proposition that the \mertensconjecture is false over $K$. This completes the discussion under falsity of the Riemann hypothesis for $\zeta_K(s)$. 

Turning to results that are instead conditional on the Riemann hypothesis for $K$, Proposition \ref{prop:zero-on-rightmost-line} in fact yields the following immediate corollary.

\begin{cor}\label{cor:RH-multiple-zero}
Assume the Riemann hypothesis for $\zeta_K(s)$, and that $\zeta_K(s)$ has a zero $\rho = \frac{1}{2} + i\gamma$, $\gamma > 0$, of multiplicity $m \geq 1$. Then
\begin{equation*}
    M_K(x) - \tilde M_K(x) = \Omega_{\pm}(x^{1/2}(\log x)^{m-1}),
\end{equation*}
where $\tilde M_K(x)$ defined as in \eqref{eq:tilde-M-K} is the part corresponding to a possible zero at $s = \frac{1}{2}$.
\end{cor}
\begin{proof}
In this case, $\frac{1}{2} = \Theta = \Theta''$ (and $= \Theta'$, if a real zero exists).
\end{proof}

\cref{cor:RH-multiple-zero} implies in particular that if $\zeta_K(s)$ satisfies the Riemann hypothesis, but possesses nontrivial zeros with multiplicity $m \geq 2$ (at $s = \frac{1}{2}$ or otherwise), then the \mertensconjecture is false. Therefore, the only non-degenerate case of the \mertensconjecture occurs when $\zeta_K(s)$ satisfies the Riemann hypothesis and has only simple nontrivial zeros (a priori none at $s = \frac{1}{2}$).

The following theorem summarizes the results of this section.

\begin{thm}\label{thm:mertens-very-fail} In the following special cases, $M_K(x)$ grows more quickly than $\sqrt x$.
\begin{enumerate}[label=(\alph*),font=\normalfont]
    
    \item If the Riemann hypothesis for $\zeta_K(s)$ fails with $\Theta=\Theta''>\Theta'$, then for all $\epsilon > 0$,
    $$M_K(x)=\Omega_{\pm}(x^{\Theta - \epsilon}).$$
    If there is a zero $\Theta+i\gamma$ of $\zeta_K(s)$, then $M_K(x)=\Omega_\pm(x^\Theta)$. 
    
    \item If the Riemann hypothesis for $\zeta_K(s)$ fails and $\Theta = \Theta'$, then
    \begin{equation*}
        M_K(x) = \Omega(x^{\Theta})
    \end{equation*}
    (in possibly only one direction).
    
    \item If the Riemann hypothesis for $\zeta_K(s)$ holds but there is a zero $\rho=\frac12+i\gamma$, $\gamma > 0$, of multiplicity $m\geq1$, then
    $$M_K(x)=\Omega\left(\sqrt x(\log x)^{m-1}\right).$$
    If $\zeta_K(\frac{1}{2})\neq 0$, or if the multiplicity of the zero of $\zeta_K$ at $s=\frac{1}{2}$ is strictly less than $m$, then this $\Omega$ can be replaced with $\Omega_\pm$.
\end{enumerate}
\end{thm}

In particular, if $M_K(x)/\sqrt x$ is bounded (both above and below), then the Riemann hypothesis for $\zeta_K(s)$ holds, and $\zeta_K(s)$ has no multiple zero in the critical strip, thereby implying that $\zeta_K(\frac12)\neq 0$. We expand on the possibility of growth in only one direction in \cref{sec:one-sided-growth}.

\begin{cor}\label{cor:galois-fields} If $K/\QQ$ is Galois with non-abelian Galois group, then the \mertensconjecture over $K$ is false.
\end{cor}
\begin{proof}
It is known that the Dedekind zeta function of a non-abelian Galois extension $K$ of $\mathbb{Q}$ has infinitely many zeros of multiplicity $\geq 2$ in the critical strip (cf.\cite{hkms}).
\end{proof}

\subsection{The explicit formula for $M_K(x)$}
We now state the explicit formula for the Mertens function, which, in its compatibility with analytic methods, is essential to our applications. We must assume the Riemann hypothesis for $\zeta_K(s)$ and that it has no nontrivial multiple zeros. The error terms in our explicit formula are presented in different cases depending on their applicability to $x > 0$; the generality of $x>0$ is necessary for the unconditional results of \cref{sec:disproof}, while more specific bounds for $x>1$ are necessary to prove the existence of a limiting distribution in \cref{sec:limdist}.

\begin{prop}\label{lem:explicit-formula} Assume the Riemann hypothesis for $\zeta_K(s)$, and that all nontrivial zeros of $\zeta_K(s)$ are simple. For any real $x>0$ and any $T\in \mathcal T$, where $\mathcal T$ is as in \cref{lem:large-horiz-strip},
$$M_K(x)=\sum_{|\gamma|\leq T} \frac{x^\rho}{\rho\zeta_K'(\rho)}+\sum_{k=0}^\infty \res_{s=-k}\frac{x^s}{s\zeta_K(s)}+E(x,T),$$
where for any $x>0$, $E(x,T)=O_x(T^{1-\epsilon})$ for any $\epsilon$, and for $x > 1$ we have
$$E(x,T)\ll\frac{x^{1+C/\log \log x}\log x}T+\frac{x}{T^{1-\epsilon}}+x^{C/\log\log x},$$
where $C$ is as in \cref{lem:dedekind-coeff-bound}. In addition, for $x > 1$,
$$\sum_{k=0}^\infty \res_{s=-k}\frac{x^s}{s\zeta_K(s)}\ll (\log x)^{O(1)},$$
which in this case can be subsumed into the error term $E(x,T)$.
\end{prop}

We defer the proof of this formula to \cref{sec:details}. 

For the previous result and others to follow, we also require the following bound on the residues of $\frac{x^s}{s\zeta_K(s)}$ at the trivial zeros $s = -k$. The purpose of this bound is two-fold: it reveals dependence on $k$ and $x$, which suffices already to prove the last statement in \cref{lem:explicit-formula}, and it clarifies dependence on $r_1$, $r_2$, and $D_K$, which is crucial in our proof of \cref{thm:mertens-genfields}.

\begin{lem}\label{lem:triv-zero-bound} Fix non-negative integers $r_1$ and $r_2$, and let $K$ be a number field with $r_1$ real embeddings and $r_2$ complex conjugate pairs of embeddings.  Then there exists some constant $c=c_{r_1,2r_2}$ such that, if $\zeta_K(s)$ has a zero at $s=-k$~for a positive integer $k$, then
$$\left|\res_{s=-k}\frac{x^s}{s\zeta_K(s)}\right|\leq \frac{c(2\pi)^{kn_K}(\log (kx\adisc_K))^{r_1 + r_2}}{kx^k\adisc_K^{k+1/2}(k!\!)^{n_K}}.$$
In addition,
$$\left|\res_{s=0}\frac{x^s}{s\zeta_K(s)}\right|\ll_K (\log x)^{r_1+r_2-1}.$$
\end{lem}

See \cref{sec:details} for the proof. Heavier consideration of the residue at $s = 0$ is necessary to~prove Theorems \ref{thm:mertens-quadfields} and \ref{thm:mertens-genfields}, wherein dependence on $n_K$ and $D_K$ will be significant.

%% file: 3jurkat.tex
In light of \cref{thm:mertens-very-fail}, if $x^{-1/2}|M_K(x)|$ is bounded, then $\zeta_K(s)$ has no zeros off the~line $\Real(s) = \frac{1}{2}$ and has only simple zeros. In this section, we describe some consequences of these assumptions, which lead to unconditional disproofs of the \mertensconjecture for ``most'' number fields. 

\subsection{The general result} In this section, we give a generic result about the limiting~behavior of $x^{-1/2}M_K(x)$ which holds unconditionally. We follow a method of Jurkat from \cite{jurkat1973mertens}, using his main theorem directly in our application. In particular, Jurkat has identified a wide class of trigonometric series which are seen to share a key limiting property enjoyed by Bohr-almost periodic functions. By exploiting this property, we show $x^{-1/2}M_K(x)$ approaches certain explicitly computable values infinitely often and arbitrarily closely.

\begin{defn}[cf.~\protect{\cite{jurkat1973mertens}}]\label{def:apd} 
A locally integrable function $f:\RR\to\RR$ is called \emph{almost periodic in a distributional sense} (APD) if there is a Bohr-almost periodic function $g$ so that $f = \frac{d^k}{d^kx}g$. 

Formally, this means $f$ admits an expansion of the form
\begin{equation}\label{eq:APD}
    f(x) \sim \sum_{n=1}^\infty \Real(a_n e^{i \lambda_n x})
\end{equation}
where $0 < \lambda_n \nearrow \infty$ and $a_n$ are complex constants such that
\begin{equation*}
    \sum_{n=1}^\infty \frac{|a_n|}{\lambda_n^k} < \infty
\end{equation*}
for some $k \in \mathbb{N}$. When referring to an APD-function, we shall specify $k$.
\end{defn}

The notion of almost periodicity in a distributional sense is useful because such functions have easily-calculable limit points, as shown by the following main theorem of Jurkat.

\begin{prop}[Jurkat~\protect{\cite[pp.~151--152]{jurkat1973mertens}}]\label{thm:apd-functions} Suppose $f:\RR\to\RR$ is an APD-function, and let $L(f)$ denote the set of Lebesgue points of $f$. Then for any $t\in L(f)$,
$$\liminf_{\substack{x\to\infty\\x\in L(f)}}f(x)\leq f(t)\leq \limsup_{\substack{x\to\infty\\x\in L(f)}}f(x);$$
Moreover, all values of $f$ at Lebesgue points are limit points of $f$. 
\end{prop}

The technical condition on Lebesgue points can be safely ignored for the remainder of this paper. The functions we work with are sufficiently well-behaved that the inclusion of their Lebesgue points does not change either limit.

For any real $x > 0$, define
\begin{equation}\label{eq:mertens-star-def}
M_K^*(x)= - \sum_{k=0}^\infty \res_{s=-k}\frac{x^s}{s\zeta_K(s)}.
\end{equation}
(We remark that this carries a slightly different normalization than Jurkat's $M^*(x)$, as he makes a few simplifications particular to the case $K=\QQ$.) Motivated by \cref{lem:explicit-formula}, this is a ``trivial completion'' of $M_K(x)$, in the sense that, for $T\in\mathcal T$ (where $\mathcal T$ is given by \cref{lem:large-horiz-strip}),
\begin{equation}\label{eq:completed-explicit}
M_K(x)+M_K^*(x)=\sum_{|\gamma|\leq T}\frac{x^\rho}{\rho\zeta_K'(\rho)}+O_x\left(\frac1{T^{1-\epsilon}}\right)
\end{equation}
for all $x>0$, assuming the preconditions of \cref{lem:explicit-formula} hold. We first show an analogue of Jurkat's main theorem for number fields under these assumptions, as well as some more stringent assumptions.

The function $f: \mathbb{R} \rightarrow \mathbb{R}$ defined by
\begin{equation*}
    f(y) \coloneqq e^{-y/2}(M_K(e^y) + M_K^*(e^y))
\end{equation*}
is locally integrable. If we assume the Riemann hypothesis for $\zeta_K(s)$ and that all nontrivial zeros are simple, and let $T \rightarrow \infty$ along $\mathcal{T}$ in the equation \eqref{eq:completed-explicit} above (with the necessary modifications for $M_K(e^y)$ if $y$ is the logarithm of an integer), it furthermore has an expansion of the form in \eqref{eq:APD} with $a_n = 2(\rho_n \zeta_K'(\rho_n))^{-1}$, $\lambda_n = \gamma_n$,
where $\gamma_n$ is the ordinate of the $n$th highest nontrivial zero $\rho_n = \frac{1}{2} + i \gamma_n$ in the upper half plane. Here, the sum converges boundedly given this grouping of the terms. Assuming for now that there exists some constant $C > 0$ such that
\begin{equation*}
    \frac{1}{|\rho \zeta_K'(\rho)|} \leq C,
\end{equation*}
the result of \cref{cor:zero-count} implies that $f$ is an APD-function with $k = 2$. Hence, \cref{thm:apd-functions} produces the following.


\begin{prop}\label{thm:m-star-limit-val} Let $K$ be a number field. Assume the Riemann hypothesis for $\zeta_K(s)$, that all nontrivial zeros of $\zeta_K(s)$ are simple, and that there exists some constant $C > 0$ for which
$$\frac{1}{|\rho\zeta_K'(\rho)|}\leq C$$
for every nontrivial zero $\rho$ of $\zeta_K(s)$. Then, for any $x>0$,
$$\frac{M_K(x)+M_K^*(x)}{x^{1/2}}$$
is a limit point of the function $x^{-1/2}M_K(x)$.
\end{prop}

We now describe our major application of this result. Recall the definition
$$M_K^-=\liminf_{x\to\infty}\frac{M_K(x)}{x^{1/2}}, \quad M_K^+=\limsup_{x\to\infty}\frac{M_K(x)}{x^{1/2}}.$$

\begin{lem}\label{lem:when-failure} Let $K$ be any number field, and let $M_K^*(x)$ be as in \eqref{eq:mertens-star-def}.
\begin{enumerate}[label=(\alph*),font=\normalfont]
    \item[(a)] Unconditionally, $M_K^+-M_K^-\geq 1$.
    \item[(b)] If $M_K^*(1)>0$ (resp.~$M_K^*(1) < -1$), then the \mertensconjecture for $K$ fails infinitely often; in particular, $M_K^+>1$ (resp.~$M_K^- < -1$).
\end{enumerate}
\end{lem}
\begin{proof} If the Riemann hypothesis for $\zeta_K(s)$ is false, or $\zeta_K(s)$ has a nontrivial zero of multiplicity greater than $1$, then \cref{thm:mertens-very-fail} implies that either $M_K^+=\infty$ or $M_K^-=-\infty$, in which case both (a) and (b) follow. By the sign of the functional equation for $\zeta_K(s)$ in \cref{thm:fe}, any zero at $s=1/2$ must be of even order, and thus any $K$ with $\zeta_K(1/2)=0$ is covered in the previous cases.

In the case where the Riemann hypothesis holds and all nontrivial zeros are simple, we may apply \cref{prop:zero-on-rightmost-line} with $\Theta''=1/2$ and $m=1$ to obtain
$$M_K^+,-M_K^-\geq \frac1{|\rho\zeta_K'(\rho)|}$$
for any nontrivial zero $\rho$ of $\zeta_K(s)$. If one of $M_K^\pm$ is infinite, then (a) and (b) hold as before. Otherwise, if $M_K^\pm$ are both finite, then there is some constant $C$ for which
$$\frac1{|\rho\zeta_K'(\rho)|}\leq C$$
for every zero $\rho$ of $\zeta_K(s)$. So, the preconditions of \cref{thm:m-star-limit-val} hold, and we have for any $x$ that
$$\frac{M_K(x)+M_K^*(x)}{x^{1/2}}$$
is between $M_K^-$ and $M_K^+$. In particular,
$$M_K^-\leq \lim_{x\to 1^-}\frac{M_K(x)+M_K^*(x)}{x^{1/2}}=M_K^*(1)<M_K^*(1)+1=\lim_{x\to 1^+}\frac{M_K(x)+M_K^*(x)}{x^{1/2}}\leq M_K^+,$$
where limits to $1^-$ and $1^+$ denote approaching $1$ from below and above, respectively. This inequality chain shows both (a) and (b).
\end{proof}

\subsection{Imaginary quadratic fields}
In this subsection, we calculate and bound $M_K^*(x)$ explicitly for imaginary quadratic fields. We pick $D>0$ such that $-D$ is a negative fundamental~discriminant and let $K=\QQ(\sqrt{-D})$ be the imaginary quadratic field with discriminant $-D$. Let $\chi_{-D}$ denote the odd quadratic character $\chi_{-D}(n)=\left(\frac{-D}n\right)$ of conductor $D$, so that $\zeta_K(s)=\zeta(s)L(s,\chi_{-D})$. Using \cref{thm:fe}, we may compute
\begin{equation}\label{eq:m-star-im-quad}
    M_K^*(x) \coloneqq \frac{2\pi}{L(1,\chi_{-D})D^{1/2}} + \sum_{k=1}^\infty \left( \frac{D}{4\pi^2} \right)^{-k - \frac{1}{2}}\frac{(-1)^k x^{-k}}{k (k!\!)^2\zeta_K(k+1)}.
\end{equation}
Via the class number formula of Dirichlet, the leading term of \eqref{eq:m-star-im-quad} is always positive. We show that, for $x = 1$ and sufficiently large $D$, it exceeds the other terms in absolute value. 
\begin{lem}\label{lem:M_K-star(1):imaginary}
If $D>0$ and $-D$ is a fundamental discriminant, and $K=\QQ(\sqrt{-D})$, then
$$M_K^*(1)\geq \frac{2\pi}{D^{1/2}}\left(\frac1{\frac12\log D+\log\log D+2+\log 2}-\exp\left(\frac{4\pi^2}D\right)+1\right).$$
In particular, if $D>307$, then $M_K^*(1)>0$.
\end{lem}
\begin{proof} We begin by bounding from below the leading term of \eqref{eq:m-star-im-quad}. The Euler product for $L(s,\chi_{-D})$ implies that it is positive for all real $s > 1$, and hence $L(1,\chi_{-D}) > 0$, knowing that it is nonzero. On the other hand,
$$L(1,\chi_{-D}) = \sum_{n\leq N} \frac{\chi_{-D}(n)}{n} + \sum_{n > N} \frac{\chi_{-D}(n)}{n} \leq 1 + \log N + \frac{2}{N} \sqrt{D} \log D,$$
where we have used partial summation and the P\'{o}lya--Vinogradov bound $\sqrt{D} \log D$ for the partial sums of the primitive character $\chi_{-D}$ modulo $D$, as well as the bound $1 + \log x$ for $\sum_{n \leq x} \frac{1}{n}$. Setting $N = 2\sqrt{D} \log D$ gives that
\begin{equation}\label{eq:im-quad-first-term-bound}
\frac{2\pi}{D^{1/2}L(1,\chi_{-D})}\geq \frac{2\pi}{D^{1/2}\left(\frac12\log D+\log\log D+2+\log 2\right)}.
\end{equation}
Now, for any real $\sigma>1$, we have $\zeta_K(\sigma)>1$ by definition, so $\zeta_K(k+1)^{-1}<1$ for every $k \geq 1$. So, we have the bound
\begin{align*}
    \left| \sum_{k=1}^\infty \left( \frac{D}{4\pi^2} \right)^{-k - \frac{1}{2}} \frac{(-1)^k}{k(k!\!)^2 \zeta_K(k+1)} \right| &\leq \frac{2\pi}{D^{1/2}} \sum_{k=1}^\infty \left( \frac{4\pi^2}{D} \right)^k \frac{1}{k (k!)^2\zeta_K(k+1)} \\
    &\leq \frac{2\pi}{D^{1/2}} \sum_{k=1}^\infty \left( \frac{4\pi^2}{D} \right)^k \frac{1}{k!} \\
    &\leq \frac{2\pi}{D^{1/2}} \left( \exp\left( \frac{4\pi^2}{D} \right) - 1 \right).
\end{align*}
Combining this bound with \eqref{eq:im-quad-first-term-bound} gives the desired result.
\end{proof}

\cref{lem:M_K-star(1):imaginary} has the following immediate consequence via \cref{lem:when-failure}.

\begin{thm}\label{thm:large-imag-quad}
Let $K$ be an imaginary quadratic extension of $\mathbb{Q}$ with discriminant $-D$ such that $D > 307$. Then the \mertensconjecture over $K$, i.e. that $|M_K(x)|~\leq x^{1/2}$, is false. In particular, $M_K^*(1) > 0$, and if the Riemann hypothesis for $\zeta_K(s)$ is true and all of its nontrivial zeros are simple, then $M_K^+ \geq 1 + M_K^*(1)$.
\end{thm}

By direct computation, we may address the smaller discriminants.

\begin{thm}\label{thm:small-imag-quad}
Let $K$ be an imaginary quadratic extension of $\mathbb{Q}$ with discriminant $-D$ such that $3<D \leq 307$. Then the \mertensconjecture over $K$ is false, with $M_K^+ > 1$. 
\end{thm}
\begin{proof}
The values of $M_K^*(1)$ are tabulated in \cref{tab:imag-quad}, where the sum has been computed up to threshold $k = 50$.
\begin{table}[h!]
\centering
\fontsize{10}{12}\selectfont
 \begin{tabular}{||c c||} 
 \hline
$D$ & $M_K^*(1)$  \\ [0.5ex] 
 \hline\hline
3 & $ -0.4851\dots$ \\ \hline
4 & $ -0.5751\dots$ \\ \hline
7 & $ -0.5303\dots$ \\ \hline
8 & $ -0.5754\dots$ \\ \hline
11 & $ -0.4722\dots$ \\ \hline
15 & $ -0.1839\dots$ \\ \hline
19 & $ 0.2704\dots$ \\ \hline
20 & $ -0.0301\dots$ \\ \hline
23 & $ -0.0137\dots$ \\ \hline
24 & $ 0.0995\dots$ \\ \hline
31 & $ 0.1227\dots$ \\ \hline
35 & $ 0.3389\dots$ \\ \hline
39 & $ 0.1300\dots$ \\ \hline
40 & $ 0.4561\dots$ \\ \hline
43 & $ 1.3179\dots$ \\ \hline
47 & $ 0.1296\dots$ \\ \hline
51 & $ 0.5597\dots$ \\ \hline
52 & $ 0.6023\dots$ \\ \hline
55 & $ 0.2377\dots$ \\ \hline
56 & $ 0.2303\dots$ \\ \hline
59 & $ 0.3538\dots$ \\ \hline
67 & $ 1.6238\dots$ \\ \hline
68 & $ 0.2836\dots$ \\ \hline 
71 & $ 0.1374\dots$ \\ \hline 
\hline
\end{tabular}
 \begin{tabular}{||c c||} 
 \hline 
$D$ & $M_K^*(1)$  \\ [0.5ex] 
\hline 
\hline
79 & $ 0.2427\dots$ \\ \hline
83 & $ 0.4594\dots$ \\ \hline
84 & $ 0.3287\dots$ \\ \hline
87 & $ 0.2050\dots$ \\ \hline
88 & $ 0.8007\dots$ \\ \hline
91 & $ 0.7820\dots$ \\ \hline
95 & $ 0.1493\dots$ \\ \hline
103 & $ 0.2864\dots$ \\ \hline
104 & $ 0.2218\dots$ \\ \hline
107 & $ 0.5178\dots$ \\ \hline
111 & $ 0.1644\dots$ \\ \hline
115 & $ 0.8403\dots$ \\ \hline
116 & $ 0.2360\dots$ \\ \hline
119 & $ 0.1296\dots$ \\ \hline
120 & $ 0.3915\dots$ \\ \hline
123 & $ 0.8609\dots$ \\ \hline
127 & $ 0.3136\dots$ \\ \hline
131 & $ 0.3007\dots$ \\ \hline
132 & $ 0.4041\dots$ \\ \hline
136 & $ 0.4047\dots$ \\ \hline
139 & $ 0.5526\dots$ \\ \hline
143 & $ 0.1434\dots$ \\ \hline
148 & $ 0.9040\dots$ \\ \hline
151 & $ 0.2231\dots$ \\ \hline 
\hline
\end{tabular}
 \begin{tabular}{||c c||} 
 \hline
$D$ & $M_K^*(1)$  \\ [0.5ex] 
\hline 
\hline
152 & $ 0.2641\dots$ \\ \hline
155 & $ 0.4156\dots$ \\ \hline
159 & $ 0.1492\dots$ \\ \hline
163 & $ 1.8941\dots$ \\ \hline
164 & $ 0.1926\dots$ \\ \hline
167 & $ 0.1367\dots$ \\ \hline
168 & $ 0.4310\dots$ \\ \hline
179 & $ 0.3342\dots$ \\ \hline
183 & $ 0.2045\dots$ \\ \hline
184 & $ 0.4368\dots$ \\ \hline
187 & $ 0.9182\dots$ \\ \hline
191 & $ 0.1181\dots$ \\ \hline
195 & $ 0.4350\dots$ \\ \hline
199 & $ 0.1815\dots$ \\ \hline
203 & $ 0.4414\dots$ \\ \hline
211 & $ 0.6023\dots$ \\ \hline
212 & $ 0.2891\dots$ \\ \hline
215 & $ 0.1127\dots$ \\ \hline
219 & $ 0.4451\dots$ \\ \hline
223 & $ 0.2483\dots$ \\ \hline
227 & $ 0.3517\dots$ \\ \hline
228 & $ 0.4550\dots$ \\ \hline
231 & $ 0.1369\dots$ \\ \hline
232 & $ 0.9499\dots$ \\ \hline 
\hline
\end{tabular}
 \begin{tabular}{||c c||} 
 \hline
$D$ & $M_K^*(1)$  \\ [0.5ex] 
\hline 
\hline
235 & $ 0.9416\dots$ \\ \hline
239 & $ 0.1076\dots$ \\ \hline
244 & $ 0.2935\dots$ \\ \hline
247 & $ 0.2999\dots$ \\ \hline
248 & $ 0.2164\dots$ \\ \hline
251 & $ 0.2471\dots$ \\ \hline
255 & $ 0.1404\dots$ \\ \hline
259 & $ 0.4537\dots$ \\ \hline
260 & $ 0.2186\dots$ \\ \hline
263 & $ 0.1300\dots$ \\ \hline
264 & $ 0.2180\dots$ \\ \hline
267 & $ 0.9542\dots$ \\ \hline
271 & $ 0.1561\dots$ \\ \hline
276 & $ 0.2200\dots$ \\ \hline
280 & $ 0.4647\dots$ \\ \hline
283 & $ 0.6232\dots$ \\ \hline
287 & $ 0.1220\dots$ \\ \hline
291 & $ 0.4633\dots$ \\ \hline
292 & $ 0.4664\dots$ \\ \hline
295 & $ 0.2253\dots$ \\ \hline
296 & $ 0.1751\dots$ \\ \hline
299 & $ 0.2204\dots$ \\ \hline
303 & $ 0.1781\dots$ \\ \hline
307 & $ 0.6279\dots$ \\ \hline
\hline
\end{tabular}
\newline
\centering\caption{Values of $M_K^*(1)$ for imaginary quadratic fields $K = \mathbb{Q}(\sqrt{-D})$ of discriminant $-D$, $D \leq 307$.}
\label{tab:imag-quad}
\end{table}
This implies the falsity of the \mertensconjecture over $K$ for all $D \leq 307$ where
\begin{equation*}
    -D \notin \{-3,-4,-7,-8,-11,-15,-20,-23\},
\end{equation*} 
since, via \cref{lem:when-failure}, we deduce from $M_K^*(1) > 0$ that $M_K^+ > 1$. The procedure can be remedied for most of those exceptional values of $-D$ by letting $x \rightarrow n^+$ for different choices of positive integers $n$ where
\begin{equation*}
    \frac{M_K(n^{+}) + M_K^*(n)}{n^{1/2}} > 1,
\end{equation*}
as represented in \cref{tab:imag-quad-larger-n} (we have taken the right-hand limits, as $n^{-1/2}M_K^*(n)$ is generally decreasing). This suffices to resolve all cases except for $D = 3$ (the field $K = \mathbb{Q}(\sqrt{-3})$) and $D = 4$ (the field $K = \mathbb{Q}(i)$).
\begin{table}[h!]
    \centering
\begin{tabular}{||c c c||}
\hline
$D$ & $n$ & $\frac{M_K(n^+)+ M_K^*(n)}{\sqrt{n}}$  \\ [0.5ex]
\hline
\hline
7 & $ 22 $ & $ 1.2138\dots$ \\ \hline
8 & $ 57 $ & $ 1.1777\dots$ \\ \hline
11 & $ 20 $ & $ 1.2923\dots$ \\ \hline
15 & $ 98 $ & $ 1.0080\dots$ \\ \hline
20 & $ 15 $ & $ 1.0076\dots$ \\ \hline
24 & $ 6 $ & $ 1.0261\dots$ \\ \hline
\hline
\end{tabular}
    \caption{Given $K=\QQ(\sqrt{-D})$ for a fundamental discriminant $-D$, the smallest positive integer $n$ for which $M_K(n^+)+M_K^*(n)>n^{1/2}$.}
    \label{tab:imag-quad-larger-n}
\end{table}

In order to resolve the case $D = 4$, we appeal to a method originally due to Ingham~\cite[Theorem~1]{Ingham} in examining the smoothed sums of $\sum_{|\gamma| <T} \frac{e^{i\gamma t}}{\rho \zeta_K'(\rho)}$ by a suitable kernel function $f(y)$. In particular, we select the kernel function of Jurkat--Peyerimhoff (see~\cite{JP,OdlyzkoteRiele}),
\begin{align*}
    f(y) = 
    \begin{cases}
    (1 - y) \cos \pi y + \frac{1}{\pi} \sin \pi y &\text{ if } 0 \leq y \leq 1, \\
    0 &\text{ if } y \geq 1,
    \end{cases}
\end{align*}
which produces the smoothed sum
\begin{equation}\label{eq:jurkat-peyerimhoff-means}
    h_{K,T}^*(t) = \sum_{|\gamma| < T} f\left( \frac{|\gamma|}{T} \right) \frac{e^{i \gamma y}}{\rho \zeta_K'(\rho)} = \sum_{|\gamma| < T} \left( \left( 1 - \frac{|\gamma|}{T} \right) \cos \left( \frac{\pi |\gamma|}{T} \right) + \frac{1}{\pi} \sin \left(\frac{\pi |\gamma|}{T} \right) \right) \frac{e^{i \gamma y}}{\rho \zeta_K'(\rho)}. 
\end{equation}
It holds for this sum that for $T > 0$ fixed,
\begin{equation*}
    M_K^- \leq \liminf_{t \rightarrow \infty} h_{K,T}^*(t) \leq \limsup_{t \rightarrow \infty} h_{K,T}^*(t) \leq M_K^+,
\end{equation*}
and it follows from the fact that $h_{K,T}^*(t)$ is almost periodic that
\begin{equation*}
    \liminf_{t \rightarrow \infty} h_{K,T}^*(t) \leq h_{K,T}^*(t) \leq \limsup_{t \rightarrow \infty} h_{K,T}^*(t),
\end{equation*}
so it suffices to exhibit $T > 0$ and $t \in \mathbb{R}$ for which $h_{K,T}^*(t) > 1$ or $h_{K,T}^*(t) < -1$.  

A computation demonstrates that $h^*_{\QQ(i), 600}(72.85) \leq -1.008$ and $h^*_{\QQ(i), 600}(-85.15) \geq 1.029$. That is, $M_K^- \leq -1.008 < 1.029 \leq M_K^+$. This completes the proof.
\end{proof}

Combining \cref{thm:large-imag-quad} and \cref{thm:small-imag-quad} concludes the proof of \cref{thm:mertens-quadfields}(a).

\subsection{Real quadratic fields} Next, we calculate and bound $M_K^*(x)$ explicitly for real quadratic fields. Pick a positive fundamental discriminant $D>0$ and let $K=\QQ(\sqrt{D})$ be the~real quadratic field with discriminant $D$. Then $\zeta_K(s)=\zeta(s)L(s,\chi_D)$, where $\chi_D$ is the even quadratic character $\chi_D(n)=\left(\frac Dn\right)$ of conductor $D$. We compute
\begin{dmath}\label{eq:m-star-real-quad}
    M_K^*(x) = \frac{4}{L(1,\chi_D)D^{1/2}} \left( \log \left( \frac{xD}{4\pi^2} \right) - \gamma + \frac{L'}{L}(1,\chi_D) \right)+ \frac{2}{D^{1/2}} {\sum_{k=1}^\infty \frac{\left(\frac{4\pi^2}{xD}\right)^{2k}}{k((2k)!\!)^2 \zeta_K(2k+1)} \left( \log \left( \frac{xD}{4\pi^2} \right) + \frac{1}{2k} + \frac{\zeta_K'}{\zeta_K}(2k+1) + 2 \frac{\Gamma'}{\Gamma}(2k+1) \right).}
\end{dmath}
Here, $\gamma = 0.5772\dots$ is Euler's constant. The $L'/L$, $\zeta_K'/\zeta_K$, and $\Gamma'/\Gamma$ terms arise (in contrast to \eqref{eq:m-star-im-quad}) since $\zeta_K(s)$ has zeros of multiplicity two when $K$ is real quadratic. We now only~need to identify which $D$ satisfy the estimate $M_K^*(1)>0$. We first provide a bound on the $L'/L$ term that is particular to the case of real characters.

\begin{lem}\label{lem:LprimeoverL1chi} Let $\chi$ be a non-principal real primitive character of conductor $D$. Then
$$\frac{L'}{L}(1,\chi)> -\frac12\left(\log\frac D{2\pi}-\gamma\right)+\frac{\chi(-1)}2\log 2.$$
\end{lem}
\begin{proof} Recall \cite[Corollary~10.18]{MV} for any Dirichlet character $\chi$ of conductor $D$ the identity
$$\frac{L'}{L}(s,\chi)=-\frac{L'}{L}(1,\overline\chi)-\frac12\frac{\Gamma'}{\Gamma}\left(\frac{s+\kappa}2\right)-\log\frac D\pi+\sum_\rho\left(\frac1{s-\rho}+\frac1\rho\right)+\frac\gamma2+(1-\kappa)\log 2,$$
where $\kappa=\frac{1-\chi(-1)}2$ and $\rho$ runs over all nontrivial zeros of $L(s,\chi)$. Specifying to $s=1$ and $\chi = \overline{\chi}$ gives
$$2\frac{L'}{L}(1,\chi)=-\frac12\frac{\Gamma'}{\Gamma}\left(\frac{1+\kappa}2\right)-\log\frac D\pi+\sum_\rho\frac1{\rho(1-\rho)}+\frac{\gamma}2+(1-\kappa)\log 2.$$
Since
$$\frac{\Gamma'}{\Gamma}\left(\frac12\right)=-\gamma-2\log 2\qquad \frac{\Gamma'}{\Gamma}(1)=-\gamma,$$
this gives
$$2\frac{L'}{L}(1,\chi)=\gamma+2(1-\kappa)\log 2-\log\frac D\pi+\sum_\rho\frac1{\rho(1-\rho)}.$$
By grouping $\rho$ and $\overline\rho$ in the sum we have
\begin{align*}
\sum_\rho\frac1{\rho(1-\rho)}
&=\frac12\sum_\rho\left(\frac1{\rho(1-\rho)}+\frac1{\overline\rho(1-\overline\rho)}\right)\\
&=\sum_\rho\frac{\Real(\rho-\rho^2)}{|\rho|^2|1-\rho|^2}\\
&=\sum_\rho\frac{\beta(1-\beta)+\gamma^2}{|\rho|^2|1-\rho|^2}>0,
\end{align*}
where $\rho=\beta+i\gamma$ and we have used $0<\beta<1$. This gives
$$2\frac{L'}{L}(1,\chi)=\gamma+2(1-\kappa)\log 2-\log\frac D\pi+\sum_\rho\frac1{\rho(1-\rho)}> \gamma+\log2+\chi(-1)\log 2-\log\frac D\pi,$$
as desired.
\end{proof}

This allows us to establish the following result on $M_K^*(1)$.

\begin{thm}
Let $K$ be a real quadratic extension of $\mathbb{Q}$ with discriminant $D > 269$. Then the \mertensconjecture over $K$, i.e. that $-1 \leq M_K^- \leq M_K^+ \leq 1$, is false. In particular, $M_K^*(1) > 0$, and if the Riemann hypothesis for $\zeta_K(s)$ is true and all of its nontrivial zeros are simple and not real, then $M_K^+ \geq 1 + M_K^*(1)$.
\end{thm}
\begin{proof}
We begin by bounding the contribution of the infinite series in $M_K^*(1)$. After dividing out by a factor of $2/D^{1/2}$, the series under consideration takes the form
\begin{equation}
    \Sigma = \sum_{k=1}^\infty \left( \frac{4\pi^2}{D} \right)^{2k} \frac{1}{k ((2k)!\!)^2 \zeta_K(2k+1)} \left( \log\left( \frac{D}{4\pi^2} \right) + \frac{1}{2k} + \frac{\zeta_K'}{\zeta_K}(2k+1) + 2 \frac{\Gamma'}{\Gamma}(2k+1) \right).
\end{equation}
We note that $\zeta_K(2k + 1) \geq 1$ for all $k \geq 1$, and by the triangle inequality,
\begin{align*}
    \left| \frac{\zeta_K'}{\zeta_K}(2k+1) \right| &= \left| \frac{\zeta'}{\zeta}(2k+1) +  \frac{L'}{L}(2k+1,\chi_D) \right| \leq 2\sum_{n=1}^\infty \frac{\Lambda(n)}{n^{2k+1}}  \\
    &\leq 2\int_1^\infty \frac{\log x}{x^{2k+1}}\, dx = \frac{1}{2k^2}.
\end{align*}
We also have the following known identity for the digamma function:
\begin{equation*}
    2\frac{\Gamma'}{\Gamma}(2k + 1) = 2 \left( \sum_{n=1}^{2k} \frac{1}{n} - \gamma \right) \leq 2 \log(2k) + (2 - 2\gamma),
\end{equation*}
where $\gamma$ is Euler's constant. Finally, we use these bounds to see that, for $D > 4\pi^2$,
\begin{align*}
    \Sigma &\leq \sum_{k=1}^\infty \left( \frac{4\pi^2}{D} \right)^{2k} \frac{1}{k ((2k)!\!)^2} \left( \log\left( \frac{D}{4\pi^2} \right) + \frac{1}{2k} + \frac{1}{2k^2} + 2 \log(2k) + (2 - 2\gamma) \right) \\
    &\leq \left( \cosh\left( \frac{4\pi^2}{D} \right) - 1 \right) \left( \frac{1}{2}\log \left( \frac{D}{4\pi^2} \right) + \frac{1}{4} + \frac{1}{4} + \log 2 + (1 - \gamma) \right)\\
    &= \left( \cosh \left( \frac{4\pi^2}{D} \right) - 1 \right) \left( \frac{1}{2} \log D +\frac 32-\log\pi -\gamma\right).
\end{align*}

We now turn our attention to the first term of $M_K^*(1)$ in \eqref{eq:m-star-real-quad}. Upon dividing out by a factor of $2/D^{1/2}$, we have
\begin{equation*}
    \frac{2}{L(1,\chi_D)} \left( \log \left( \frac{D}{4\pi^2} \right) - \gamma + \frac{L'}{L}(1,\chi_D) \right).
\end{equation*}
To this end, we recall the previously obtained bound from \eqref{eq:im-quad-first-term-bound}
\begin{equation*}
    \frac{1}{L(1,\chi_D)} \geq \frac{1}{\frac{1}{2} \log D + \log \log D + 2 + \log 2}.
\end{equation*}
Finally, using \cref{lem:LprimeoverL1chi} in the case of an even character, the problem reduces to finding the values $D$ at which the inequality
\begin{dmath*}
    \frac{2 \left(\frac12\log D - \frac32\log\pi-\log2-\frac\gamma2\right)}{\frac{1}{2} \log D + \log \log D + 2 + \log 2}
    > \left( \cosh\left( \frac{4\pi^2}{D} \right) - 1 \right) \left( \frac{1}{2} \log D + \frac32-\gamma-\log\pi \right)
\end{dmath*}
holds. This indeed holds for all $D > 269$.
\end{proof}

By direct computation, we may address the smaller discriminants.

\begin{thm}\label{thm:small-real-quad}
Let $K$ be an real quadratic extension of $\mathbb{Q}$ with discriminant $5 < D \leq 269$. Then the \mertensconjecture over $K$ is false, with $M_K^+ > 1$.

\begin{proof}
As in the proof of \cref{thm:small-imag-quad}, the values of $M_K^*(1)$ are tabulated in \cref{tab:real-quad}, where the sum has been computed up to threshold $k = 50$. 
\begin{table}[h!]
\centering
\fontsize{10}{12}\selectfont
\begin{tabular}{||c c||}
\hline
$D$ & $M_K^*(1)$  \\ [0.5ex]
\hline
\hline
5 & $ -0.4857\dots$ \\ \hline
8 & $ -0.5362\dots$ \\ \hline
12 & $ -0.5230\dots$ \\ \hline
13 & $ -0.6401\dots$ \\ \hline
17 & $ -0.3642\dots$ \\ \hline
21 & $ -0.5180\dots$ \\ \hline
24 & $ -0.3364\dots$ \\ \hline
28 & $ -0.2605\dots$ \\ \hline
29 & $ -0.3697\dots$ \\ \hline
33 & $ -0.1707\dots$ \\ \hline
37 & $ -0.2060\dots$ \\ \hline
40 & $ -0.1436\dots$ \\ \hline
41 & $ -0.1227\dots$ \\ \hline
44 & $ -0.1153\dots$ \\ \hline
53 & $ 0.1152\dots$ \\ \hline
56 & $ -0.0201\dots$ \\ \hline
57 & $ -0.0452\dots$ \\ \hline
60 & $ -0.0198\dots$ \\ \hline
61 & $ -0.0162\dots$ \\ \hline
65 & $ -0.0135\dots$ \\ \hline
69 & $ 0.0750\dots$ \\ \hline
\hline
\end{tabular}
\begin{tabular}{||c c||}
\hline
$D$ & $M_K^*(1)$  \\ [0.5ex]
\hline
\hline
73 & $ -0.0079\dots$ \\ \hline
76 & $ 0.0126\dots$ \\ \hline
77 & $ 0.4411\dots$ \\ \hline
85 & $ 0.0816\dots$ \\ \hline
88 & $ 0.0506\dots$ \\ \hline
89 & $ 0.0410\dots$ \\ \hline
92 & $ 0.2036\dots$ \\ \hline
93 & $ 0.2430\dots$ \\ \hline
97 & $ 0.0274\dots$ \\ \hline
101 & $ 0.3760\dots$ \\ \hline
104 & $ 0.1804\dots$ \\ \hline
105 & $ 0.0461\dots$ \\ \hline
109 & $ 0.1101\dots$ \\ \hline
113 & $ 0.0921\dots$ \\ \hline
120 & $ 0.1362\dots$ \\ \hline
124 & $ 0.0851\dots$ \\ \hline
129 & $ 0.0617\dots$ \\ \hline
133 & $ 0.2189\dots$ \\ \hline
136 & $ 0.0951\dots$ \\ \hline
137 & $ 0.1183\dots$ \\ \hline
140 & $ 0.2922\dots$ \\ \hline
\hline
\end{tabular}
\begin{tabular}{||c c||}
\hline
$D$ & $M_K^*(1)$  \\ [0.5ex]
\hline
\hline
141 & $ 0.2217\dots$ \\ \hline
145 & $ 0.0569\dots$ \\ \hline
149 & $ 0.3829\dots$ \\ \hline
152 & $ 0.4500\dots$ \\ \hline
156 & $ 0.1444\dots$ \\ \hline
157 & $ 0.2747\dots$ \\ \hline
161 & $ 0.1068\dots$ \\ \hline
165 & $ 0.3070\dots$ \\ \hline
168 & $ 0.2387\dots$ \\ \hline
172 & $ 0.1414\dots$ \\ \hline
173 & $ 1.2271\dots$ \\ \hline
177 & $ 0.0981\dots$ \\ \hline
181 & $ 0.1952\dots$ \\ \hline
184 & $ 0.1087\dots$ \\ \hline
185 & $ 0.1436\dots$ \\ \hline
188 & $ 0.5386\dots$ \\ \hline
193 & $ 0.0742\dots$ \\ \hline
197 & $ 0.8499\dots$ \\ \hline
201 & $ 0.0901\dots$ \\ \hline
204 & $ 0.1653\dots$ \\ \hline
& \\
\hline \hline
\end{tabular}
\begin{tabular}{||c c||}
\hline
$D$ & $M_K^*(1)$  \\ [0.5ex]
\hline
\hline
205 & $ 0.2148\dots$ \\ \hline
209 & $ 0.1307\dots$ \\ \hline
213 & $ 0.6172\dots$ \\ \hline
217 & $ 0.0825\dots$ \\ \hline
220 & $ 0.1559\dots$ \\ \hline
221 & $ 0.3885\dots$ \\ \hline
229 & $ 0.2186\dots$ \\ \hline
232 & $ 0.1609\dots$ \\ \hline
233 & $ 0.1703\dots$ \\ \hline
236 & $ 0.3216\dots$ \\ \hline
237 & $ 0.6833\dots$ \\ \hline
241 & $ 0.0724\dots$ \\ \hline
248 & $ 0.6696\dots$ \\ \hline
249 & $ 0.0898\dots$ \\ \hline
253 & $ 0.2825\dots$ \\ \hline
257 & $ 0.2055\dots$ \\ \hline
264 & $ 0.2125\dots$ \\ \hline
265 & $ 0.0824\dots$ \\ \hline
268 & $ 0.1677\dots$ \\ \hline
269 & $ 0.5460\dots$ \\ \hline
& \\
\hline \hline
\end{tabular}
\newline
\centering\caption{Values of $M_K^*(1)$ for imaginary quadratic fields $K = \mathbb{Q}(\sqrt{D})$ of discriminant $D$, $D \leq 269$.}
\label{tab:real-quad}
\end{table}
This implies the falsity of the \mertensconjecture over $K$ for all $D \leq 269$ where
\begin{equation*}
    D \notin \{5,8,12,13,17,21,24,28,29,33,37,40,41,44,56,57,60,61,65,73\}.
\end{equation*} 
Again, this can be remedied for most of the above discriminants by letting $x \rightarrow n^{+}$ for other positive integers $n$ for which $n^{-1/2}(M_K(n^{+})+ M_K^*(n)) > 1$. These findings are presented in \cref{tab:real_quadratic_exceptional_cases}, which suffices to resolve all cases except for $D = 5$ (the field $\QQ(\sqrt{5})$), $D = 8$ (the field $\QQ(\sqrt{2})$), and $D = 12$ (the field $\QQ(\sqrt{3})$).
\begin{table}[h!]
    \centering
\begin{tabular}{||c c c||}
\hline
$D$ & $n$ & $\frac{M_K(n^+) + M_K^*(n)}{\sqrt{n}}$  \\ [0.5ex]
\hline
\hline
13 & $ 100 $ & $ 1.0670\dots$ \\ \hline
17 & $ 38 $ & $ 1.0439\dots$ \\ \hline
21 & $ 35 $ & $ 1.2558\dots$ \\ \hline
24 & $ 146 $ & $ 1.1272\dots$ \\ \hline
28 & $ 94 $ & $ 1.0049\dots$ \\ \hline
29 & $ 49 $ & $ 1.0038\dots$ \\ \hline
\hline
\end{tabular}
\begin{tabular}{||c c c||}
\hline
$D$ & $n$ & $\frac{M_K(n^+) + M_K^*(n)}{\sqrt{n}}$  \\ [0.5ex]
\hline
\hline
33 & $ 82 $ & $ 1.0991\dots$ \\ \hline
37 & $ 33 $ & $ 1.4651\dots$ \\ \hline
40 & $ 159 $ & $ 1.0713\dots$ \\ \hline
41 & $ 215 $ & $ 1.2509\dots$ \\ \hline
44 & $ 917 $ & $ 1.0343\dots$ \\ \hline
56 & $ 65 $ & $ 1.1612\dots$ \\ \hline
\hline
\end{tabular}
\begin{tabular}{||c c c||}
\hline
$D$ & $n$ & $\frac{M_K(n^+) + M_K^*(n)}{\sqrt{n}}$  \\ [0.5ex]
\hline
\hline
57 & $ 146 $ & $ 1.1290\dots$ \\ \hline
60 & $ 35 $ & $ 1.2919\dots$ \\ \hline
61 & $ 39 $ & $ 1.1080\dots$ \\ \hline
65 & $ 26 $ & $ 1.0031\dots$ \\ \hline
73 & $ 9 $ & $ 1.1778\dots$ \\ \hline
 & & \\
\hline \hline
\end{tabular}
    \caption{Given $K=\QQ(\sqrt{-D})$ for a fundamental discriminant $-D$, the smallest positive integer $n$ for which $M_K(n^+)+M_K^*(n)>n^{1/2}$.}
    \label{tab:real_quadratic_exceptional_cases}
\end{table} 
For these fields, we return to the means $h_{K,T}^*(t)$ as defined in \eqref{eq:jurkat-peyerimhoff-means}. A computation shows that $t = 17.32$ at $T = 200$ suffices for $K = \QQ(\sqrt{3})$ with a bound of $M_K^+ > 1.027$. Moreover, $t = -24.64$ at $T = 150$ suffices for $K = \QQ(\sqrt{2})$ with a bound of $M_K^+ > 1.049$.
\end{proof}
\end{thm}

This concludes the proof of \cref{thm:mertens-quadfields}(b).

\subsection{General number fields}
In this subsection, we show that $M_K^*(1)>0$ for all number fields $K$ with sufficiently large discriminant and fixed signature. In light of \cref{lem:when-failure}, this implies that the \mertensconjecture is false for such $K$. We first show that the contribution of the residues at strictly negative integers is relatively small.

\begin{lem}\label{lem:residues-negative-ints-bounds} Fix a signature $(r_1,r_2)$, and consider a number field $K$ with $r_1$ real embeddings, $2r_2$ complex embeddings, and varying discriminant $\disc_K$ of magnitude $\adisc_K$. Then for all $x \geq 1$,
$$\sum_{k=1}^\infty \res_{s=-k}\frac{x^s}{s\zeta_K(s)}\ll \frac{(\log(x\adisc_K))^{n_K}}{x\adisc_K^{3/2}}.$$
\end{lem}
\begin{proof} Apply \cref{lem:triv-zero-bound} and sum the results.
\end{proof}

The remainder of this section will involve bounding the residue at $s=0$. In our application, we are taking $x = 1$, and it is easy to see that
\begin{equation*}
    \res_{s = 0} \frac{x^s}{s\zeta_K(s)} \bigg|_{x = 1} = \res_{s = 0} \frac{1}{s\zeta_K(s)}.
\end{equation*}
We first give a result about the Laurent series coefficients of $\zeta_K'(s)/\zeta_K(s)$ at $s = 1$. 

\begin{lem}\label{lem:high-deriv-log-deriv} Let $K$ be a number field of fixed degree, and assume the Riemann hypothesis for $\zeta_K(s)$. Then, for a fixed integer $j\geq 0$,
$$\lim_{s\to 1}\frac{d^j}{ds^j}\left(\frac1{s-1}+\frac{\zeta_K'(s)}{\zeta_K(s)}\right)\ll(\log\log\adisc_K)^{j+1}$$
\end{lem}
\begin{proof} We mimic the proof of \cite[Theorem 13.13]{MV}. Fix some real $x,y\geq 2$, to be chosen later. Define $\Lambda^K(n)$ to be the coefficient of $n^{-s}$ in the Dirichlet series expansion of $\zeta_K'(s)/\zeta_K(s)$, and define
$$w(u)=w(x, y; u) \coloneqq \begin{cases}
	1 & 1 \leq u \leq x \\
	1 - \frac{\log u/x}{\log y} & x \leq u \leq xy\\
	0 & u \geq xy.
\end{cases}$$
By Riesz summation (\hspace{1sp}\cite[(5.20)]{MV}; cf. \cite[Theorem 13.13]{MV}), we have
\begin{equation}\label{eq:explicit-w-sum}
\sum_{n\leq xy}\frac{w(n)\Lambda^K(n)}{n^s}=-\frac{\zeta_K'(s)}{\zeta_K(s)}+\frac{(xy)^{1-s}-x^{1-s}}{(1-s)^2\log y}-\sum_\rho\frac{(xy)^{\rho-s}-x^{\rho-s}}{(\rho-s)^2\log y},
\end{equation}
where the sum runs over all zeros of $\zeta_K(s)$, both trivial and nontrivial, and a zero of multiplicity $m$ appears $m$ times. Now, as the nontrivial zeros occur only at non-positive integers and with multiplicity bounded by $n_K$
\begin{equation}\label{eq:trivial-zero-sum}
\left|\sum_{\rho\text{ trivial}}\frac{(xy)^{\rho-s}-x^{\rho-s}}{(\rho-s)^2\log y}\right|\leq n_K\sum_{k=0}^\infty \left|\frac{(xy)^{-k-s}-x^{-k-s}}{(k+s)^2\log y}\right|\ll \frac{2x^{-3/4}}{\log y},
\end{equation}
where the last bound follows from $|s-1|~\leq \frac{1}{4}$. Also,
\begin{align*}
\left|\sum_{\rho\text{ nontrivial}}\frac{(xy)^{\rho-s}-x^{\rho-s}}{(\rho-s)^2\log y}\right|
&\leq 2\sum_{T=0}^\infty \sum_{T\leq \gamma<T+1}\frac{x^{1/2-\sigma}\left|y^{\rho-s}-1\right|}{\left|\frac12-\sigma+i(\gamma-t)\right|^2\log y}\\
&\leq \frac{4x^{-1/4}}{\log y}\sum_{T=0}^\infty \frac{N(T+1)-N(T)}{\frac1{16}+T^2}\ll \frac{x^{-1/4}\log\adisc_K}{\log y},
\end{align*}
where we have used \cref{cor:zero-count} in the last simplification. As a result,
$$f(s)\coloneqq \frac{\zeta_K'(s)}{\zeta_K(s)}-\left[-\sum_{n\leq xy}\frac{w(n)\Lambda^K(n)}{n^s}+\frac{(xy)^{1-s}-x^{1-s}}{(1-s)^2\log y}\right]\ll O\left(\frac{x^{-1/4}\log\adisc_K}{\log y}\right)$$
for $|s-1|~\leq \frac{1}{4}$. Now, the Taylor series expansion
$$\frac{\zeta_K'(s)}{\zeta_K(s)}=-\frac1{s-1}+\sum_{j=0}^\infty \frac{a_j(s-1)^j}{j!\!},\qquad a_j=\lim_{s\to 1}\frac{d^j}{ds^j}\left(\frac1{s-1}+\frac{\zeta_K'(s)}{\zeta_K(s)}\right)$$
converges for $|s-1|~<\frac{1}{2}$, since $\zeta_K(s)$ has no zeros strictly within distance $\frac{1}{2}$ of $s=1$. Furthermore, we have the expansions
$$\sum_{n\leq xy}\frac{w(n)\Lambda^K(n)}{n^s}=\sum_{j=0}^\infty \frac{(s-1)^j}{j!\!}\sum_{n\leq xy}\frac{w(n)\Lambda^K(n)(-\log n)^j}{n}$$
and
$$\frac{(xy)^{1-s}-x^{1-s}}{(1-s)^2\log y}=-\frac1{s-1}+\sum_{j=0}^\infty \frac{(s-1)^j}{j!\!}\left(\frac{(-1)^j\left(\log(xy)^{j+2}-(\log x)^{j+2}\right)}{(j+1)(j+2)\log y}\right).$$
Now, select $x=(\log\adisc_K)^4$ and $y=e$, so that $f(s)=O(1)$. Writing
$$f(s)=\sum_{j=0}^\infty \frac{b_j(s-1)^j}{j!\!},$$
we have
$$b_j=j!\! \oint_{\Gamma_r} f(s)(s-1)^{-j-1}ds=\frac1{2\pi}\int_0^{2\pi}f\left(1+re^{i\theta}\right)r^{-j-1}e^{-i(j+1)\theta}d\theta$$
for any $r\leq \frac{1}{4}$, where $\Gamma_r$ is a counterclockwise circular contour of radius $r$ around $1$. For fixed $j$ and $r$, bounding the integrand in the second integral gives $b_j=O(1)$. So, recalling our choices of $x$ and $y$ and using partial summation gives
$$a_j=\sum_{n\leq xy}\frac{w(n)\Lambda^K(n)(-\log n)^j}{n}+\frac{(-1)^j\left(\log(xy)^{j+2}-(\log x)^{j+2}\right)}{(j+1)(j+2)\log y}+O(1)\ll(\log\log\adisc_K)^{j+1}.$$
This completes the proof. 
\end{proof}

We are now ready to bound the residue at $0$.

\begin{lem}\label{lem:residue-at-zero} In the same setting as \cref{lem:residues-negative-ints-bounds},
$$-\res_{s=0}\frac1{s\zeta_K(s)}=\frac{2^{r_1+r_2}\pi^{r_2}(\log\adisc_K)^{r_1+r_2-1}}{\adisc_K^{1/2}\lim_{s\to 1}{(s-1)\zeta_K(s)}}(1+o(1))\gg_{r_1,r_2} \frac{(\log\adisc_K)^{-r_2}}{D_K^{1/2}}.$$
\end{lem}
\begin{proof} Write
$$\res_{s=0}\frac1{s\zeta_K(s)}=\lim_{s\to 0}\frac{d^{r_1+r_2-1}}{ds^{r_1+r_2-1}}\left(\frac{s^{r_1+r_2-1}}{\zeta_K(s)}\right).$$
Using the functional equation, we may write
$$\frac{s^{r_1+r_2-1}}{\zeta_K(s)}=\left(\frac{2^{r_2}\pi^{n_K}}{\adisc_K^{1/2}}\right)\left(\frac{\adisc_K}{(2\pi)^{n_K}}\right)^s\left(\frac1{s\zeta_K(1-s)}\right)\left(\frac{s}{\sin\frac{\pi s}2}\right)^{r_1}\left(\frac s{\sin\pi s}\right)^{r_2}\left(\frac1{\Gamma(1-s)}\right)^{n_K}.$$
Define
\begin{align*}
f(s) &=\frac{s^{r_1+r_2-1}}{\zeta_K(s)}\\
g_1(s)&=s\log\left(\frac{\adisc_K}{(2\pi)^{n_K}}\right)\\
g_2(s)&=\log(s\zeta_K(1-s))\\
g_3(s)&=\log s-\log\sin\frac{\pi s}2\\
g_4(s)&=\log s-\log\sin\pi s\\
g_5(s)&=\log\Gamma(1-s)
\end{align*}
so that
$$g(s)\coloneqq\log f(s)=\log\left(\frac{2^{r_2}\pi^{n_K}}{\adisc_K^{1/2}}\right)+g_1(s)-g_2(s)+r_1g_3(s)+r_2g_4(s)-n_Kg_5(s).$$
Since $f'(s)=g'(s)f(s)$, we have via induction on $j$ that
$$f^{(j)}(s)=f(s)P_j\left(g^{(1)}(s),g^{(2)}(s),\dots,g^{(j)}(s)\right)$$
for some polynomial $P_j(t_1,\dots,t_j)$ in $j$ variables with positive integral coefficients, which is homogeneous of degree $j$ when $t_i$ is assigned degree $i$. In addition, the coefficient of $t_1^j$ in $P_j(t_1,\dots,t_j)$ is $1$. We now bound the derivatives of $g$ at $1$.

For each $1\leq j\leq r_1+r_2-1$, the derivatives $g_i^{(j)}(s)$ for $i\in\{3,4,5\}$ near $s=0$ depend only on $r_1$ and $r_2$, so we can treat the terms corresponding to these $g_i$ as constants. This gives that, for $j\geq 1$,
$$g^{(j)}(s)=-g_2^{(j)}(s)+O(1)+\begin{cases}\log\adisc_K&j=1\\0&j> 1.\end{cases}$$
Now, note that
$$g_2^{(j)}(s)=\frac{d^{j-1}}{ds^{j-1}}\left(\frac1s-\frac{\zeta_K'(1-s)}{\zeta_K(1-s)}\right),$$
so \cref{lem:high-deriv-log-deriv} shows that $g_2^{(j)}(0)\ll(\log\log\adisc_K)^{j}$ for all $j\geq 1$. This implies that, using the properties of $P_j$,
$$\lim_{s\to 0}P_{r_1+r_2-1}\left(g^{(1)}(s),g^{(2)}(s),\dots,g^{(r_1+r_2-1)}(s)\right)=(\log \adisc_K)^{r_1+r_2-1}(1+o(1)).$$
Also,
$$\lim_{s\to 0}f(s)=\left(\frac{2^{r_2}\pi^{n_K}}{\adisc_K^{1/2}}\right)\left(\lim_{s\to 0}\frac1{s\zeta_K(1-s)}\right)\left(\frac2\pi\right)^{r_1}\left(\frac1\pi\right)^{r_2}=-\frac{2^{r_1+r_2}\pi^{r_2}}{\adisc_K^{1/2}}\lim_{s\to 1}\frac1{(s-1)\zeta_K(s)}.$$
By a theorem of Louboutin \cite[Theorem 1]{LOUBOUTIN2000263},
$$\lim_{s\to 1}(s-1)\zeta_K(s)\ll (\log \adisc_K)^{n_K-1},$$
and is clearly positive. This gives the desired result.
\end{proof}

Using this bound, we finish the proof that $M_K^*(1)>0$ for all sufficiently large $D_K$.

\begin{proof}[Proof of \cref{thm:mertens-genfields}]
By \cref{lem:residues-negative-ints-bounds},
$$M_K^*(1)=-\res_{s=0}\frac1{s\zeta_K(s)}+O\left(\frac{(\log \adisc_K)^{n_K}}{\adisc_K^{3/2}}\right).$$
Now, \cref{lem:residue-at-zero} implies that the first term is positive, and that it grows at least as quickly as 
$$\frac{(\log\adisc_K)^{-r_2}}{\adisc_K^{1/2}}.$$
For large $\adisc_K$ depending on $r_1$ and $r_2$, this outpaces the other terms, and so $M_K^*(1)>0$. An application of \cref{lem:when-failure} finishes the proof.
\end{proof}

%% file: 4limdist.tex
In this section, we fix a number field $K$ whose Dedekind zeta function satisfies the Riemann hypothesis and $J_{-1}^K(T) \ll_{\alpha} T^{1+\alpha}$ for some $0 \leq \alpha < 2 - \sqrt{3}$. We expect these hypotheses to hold for all $K$ with $K/\QQ$ an abelian Galois extension, and provide support for the second condition later in this section.

Ng~\cite{Ng} demonstrates the existence of a limiting distribution for $e^{-y/2}M(e^y)$ by achieving bounds for the error term $E(x,T)$ in the truncated explicit formula
$$M(x) = \sum_{|\gamma| \leq T} \frac{x^\rho}{\rho \zeta'(\rho)} + E(x,T)$$
for $M(x)$. He then follows the method of Rubinstein--Sarnak~\cite{RS} to prove the existence of a limiting distribution by first constructing measures for the main term at height $T$, then bounding the error term $E(x,T)$ to show that they converge to the desired probability measure.

Later work of Akbary--Ng--Shahabi in \cite{ANS} allows for substantial generalization of this method by proving that a more general class of functions of interest possess limiting distributions; we will use their framework in our proof. First, it requires the definition of a rather broad class of almost periodic functions.
\begin{defn}
Denote by $\mathbb {T}$ the class of all real-valued trigonometric polynomials
\begin{equation*}
    P_N(y) = \sum_{n=1}^N r_n e^{i \lambda_n y}, \quad y \in \mathbb{R}
\end{equation*}
where $r_n \in \mathbb{C}$ and $\lambda_n \in \mathbb{R}$. We say that a locally square-integrable function $\phi \in L_{\text{loc}}^2([0,\infty))$ is $B^2$\textit{-almost periodic} if for any $\epsilon > 0$, there is a function $f_\epsilon(y) \in \mathbb {T}$ such that
\begin{equation*}
    \lVert \phi(y) - f_\epsilon(y) \rVert_{B^2} < \epsilon,
\end{equation*}
where
\begin{equation*}
    \lVert \phi \rVert_{B^2} \coloneqq \left( \limsup_{Y \rightarrow \infty} \frac{1}{Y} \int_0^Y |\phi(y)|^2 \, dy \right)^{1/2}.
\end{equation*}
\end{defn}

All $B^2$-almost periodic functions possess a limiting distribution \cite[Theorem~2.9]{ANS}.~This allows for concrete descriptions of almost periodic functions in the line of Rubinstein--Sarnak, as well as the later Mertens-specific work of Ng. 

\begin{thm}[Akbary--Ng--Shahabi~\protect{\cite[Corollary~1.3]{ANS}}]\label{thm:ANS}
Let $\phi:[0,\infty) \rightarrow \mathbb{R}$ be a scalar-valued function and let $y_0$ be such that $\phi$ is square-integrable on $[0,y_0]$. Assume there exists a non-decreasing sequence of positive numbers $\{\lambda_n\}_{n \in \mathbb{N}}$ which tends to infinity, a sequence of complex numbers $\{r_n\}_{n \in \mathbb{N}}$ and a real constant $c$ such that for $y \geq y_0$,
\begin{equation}\label{eq:phi-def}
    \phi(y) = c + \Real\left( \sum_{\lambda_n \leq X} r_n e^{i\lambda_n y} \right) + \mathcal{E}(y,X)
\end{equation}
for any $X \geq X_0 > 0$, where $\mathcal{E}(y,X)$ satisfies
\begin{equation}\label{eq:error-bound}
    \lim_{Y \rightarrow \infty} \frac{1}{Y} \int_{y_0}^Y |\mathcal{E}(y,e^Y)|^2 \, dy = 0.
\end{equation}
Assume that
\begin{equation}\label{eq:zero-count-one}
        \sum_{T < \lambda_n \leq T+1} 1 \ll \log T,
\end{equation}
and that either
\begin{enumerate}[label=(\alph*), font=\normalfont]
    \item $r_n \ll \lambda_n^{-\beta}$ for some $\beta > \frac{1}{2}$, or that
    \item for some $0\leq \theta<3-\sqrt 3$,
    \begin{equation}\label{eq:coefs-small}
        \sum_{\lambda_n \leq T} \lambda_n^2 |r_n|^2 \ll T^\theta.
    \end{equation}
\end{enumerate}
Then $\phi(y)$ is a $B^2$-almost periodic function and therefore possesses a limiting distribution.
\end{thm}

Therefore, to establish the existence of a logarithmic limiting distribution for $x^{-1/2}M_K(x)$, it will suffice to provide a suitable expression for
$$\phi_K(y)\coloneqq e^{-y/2}M_K\left(e^y\right)$$
as above with error small in the sense of \eqref{eq:error-bound}, as well as to verify that \eqref{eq:coefs-small} holds for some $\theta<3-\sqrt3$. Our first step will be to extend the explicit formula (\cref{lem:explicit-formula}) for $M(x)$ to all heights $T\geq 2$, assuming the analogue of the Gonek--Hejhal conjecture for $K$, as in \eqref{eq:gonek-hejhal-nfield}.
\begin{lem}\label{lem:explicit-formula-gh}
Assume the Riemann hypothesis for $\zeta_K(s)$ and that $J_{-1}^K(T)\ll_\alpha T^{1+\alpha}$.
For $x \geq 2$ and arbitrary $T \geq 2$,
\begin{equation}
    M_K(x) = \sum_{|\gamma| \leq T} \frac{x^{\rho}}{\rho \zeta_K'(\rho)} + E(x,T)
\end{equation}
where
\begin{equation*}
    E(x,T) \ll \frac{x^{1+C/\log\log x}\log x}T+\frac x{T^{1-\epsilon}}+x^{C/\log\log x}+\left(\frac{x\log T}{T^{1 - \alpha}}\right)^{1/2}
\end{equation*}
for any $\epsilon>0$ where the constant $C$ is as in \cref{lem:dedekind-coeff-bound}.
\end{lem}
\begin{proof}
Suppose $T \geq 2$ satisfies $n \leq T \leq n+1$ for some positive integer $n$, and let $T_n$ be as in \cref{lem:large-horiz-strip}. Suppose without loss of generality that $n \leq T_n \leq T \leq n+1$. \cref{lem:explicit-formula} gives that
\begin{equation*}
    M_K(x) = \sum_{|\gamma| \leq T} \frac{x^{\rho}}{\rho \zeta_K'(\rho)} - \sum_{T_n \leq |\gamma| \leq T} \frac{x^{\rho}}{\rho \zeta_K'(\rho)} + E(x,T_n),
\end{equation*}
where
$$E(x,T_n)\ll \frac{x^{1+C/\log\log x}\log x}T+\frac x{T^{1-\epsilon}}+x^{C/\log\log x}.$$
We use the assumption $J_{-1}^K(T)\ll T^{1 + \alpha}$ to bound the second sum. The Cauchy--Schwarz inequality gives that, under the Riemann hypothesis,
\begin{equation*}
    \left| \sum_{T_n \leq |\gamma| \leq T} \frac{x^{\rho}}{\rho \zeta_K'(\rho)} \right| \leq x^{1/2} \left( \sum_{T_n \leq |\gamma| \leq T} \frac{1}{|\rho \zeta_K'(\rho)|^2} \right)^{1/2} \left( \sum_{T_n \leq |\gamma| \leq T} 1 \right)^{1/2}.
\end{equation*}
Since
$$\sum_{T_n\leq|\gamma|\leq T}\frac1{|\rho\zeta_K'(\rho)|^2}\ll \frac1{T^2}\sum_{T_n\leq|\gamma|\leq T}\frac1{|\zeta_K'(\rho)|^2}\ll \frac{J_{-1}^K(T)}{T^2}\ll T^{\alpha - 1},$$
and the number of zeros of $\zeta_K$ with ordinates between $T_n$ and $T$ is $\ll \log T$ by \cref{cor:zero-count},
$$\left| \sum_{T_n \leq |\gamma| \leq T} \frac{x^{\rho}}{\rho \zeta_K'(\rho)} \right| \ll \left( \frac{x \log T}{T^{1 - \alpha}} \right)^{1/2},$$
which is enough. The case $T \leq T_n$ is treated similarly; all that differs is the sign on the zeros with ordinates between $T$ and $T_n$.
\end{proof}

We are ready to establish the main theorem.

\begin{proof}[Proof of Theorem \ref{thm:limiting-distribution-abelian}]
Setting $x = e^y$ in \cref{lem:explicit-formula-gh}, where $y \geq e$, we may define
\begin{equation*}
    \phi_K(y) = e^{-\frac{y}{2}}M_K(e^y) = \sum_{|\gamma| \leq T} \frac{e^{i \gamma y}}{\rho \zeta_K'(\rho)} + \mathcal{E}(e^y,T)
\end{equation*}
where
\begin{equation*}
    \mathcal{E}(y,T) \coloneqq e^{-\frac{y}{2}}E(e^y,T) \ll \frac{e^{\frac{Cy}{\log y} + \frac{y}{2}} (\log y)}{T} + \frac{e^{\frac{y}{2}}}{T^{1-\epsilon}} + \frac{(\log T)^{1/2}}{T^{(1 - \alpha)/2}} + e^{\frac{Cy}{\log y} - \frac{y}{2}}. 
\end{equation*}
This is of the form required in \cref{thm:ANS}; if we order the nontrivial zeros of $\zeta_K$ nondecreasingly by magnitude of ordinate (both with positive and negative ordinate) as $\rho_n=\frac12+i\gamma_n$, then our definition of $\phi_K(y)$ agrees with \eqref{eq:phi-def} with
$$\lambda_n=\gamma_n,\ r_n=\frac1{\rho_n\zeta_K'(\rho_n)}.$$
As a result, we need only to verify \eqref{eq:error-bound}, \eqref{eq:zero-count-one} and \eqref{eq:coefs-small} for some $\theta<3-\sqrt 3$. The second of these is implied by \cref{cor:zero-count}, and the third holds for $\theta=1 + \alpha$ since
$$\sum_{\lambda_n \leq T} \lambda_n^2 |r_n|^2 = \sum_{|\gamma|\leq T} \frac{4\gamma^2}{|\rho \zeta_K'(\rho)|^2}\ll J_{-1}^K(T) \ll T^{1+\alpha},$$
using our assumption on $J_{-1}^K(T)$, having chosen $0 \leq \alpha < 2 - \sqrt{3} = 0.2679\dots$. Finally, to verify \eqref{eq:error-bound}, we first have, using the Cauchy--Schwarz inequality, that
\begin{equation*}
    \int_e^Y |\mathcal{E}(y,e^Y)|^2 \, dy \ll \int_e^Y \left( \frac{e^{\frac{2Cy}{\log y} + y} (\log y)^2}{e^{2Y}} + \frac{e^{y}}{e^{2Y(1-\epsilon)}} + \frac{Y}{e^{Y(1-\alpha)}} + e^{\frac{2Cy}{\log y} - y} \right) \, dy.
\end{equation*}
The integral of the third term is bounded as $Y \rightarrow \infty$ when $\alpha < 2 - \sqrt{3}$, and so is that of the second term when $\epsilon< 1/2$. Moreover, an application of L'H\^opital's rule shows that the first and fourth terms are $o(Y)$. This gives
$$\lim_{Y \rightarrow \infty} \frac{1}{Y} \int_e^Y |\mathcal{E}(y,e^Y)|^2 \, dy = 0,$$
as desired. We may conclude that $\phi_K(y)=e^{-y/2}M_K(e^y)$ possesses a limiting distribution by \cref{thm:ANS}.
\end{proof}

\subsection{Support for the conjecture \eqref{eq:gonek-hejhal-nfield}}
By the discussion following the statement of \cref{thm:limiting-distribution-abelian}, the picture is most transparent when $K/\mathbb{Q}$ is an abelian Galois extension. In this case, the irreducible characters on $\gal(K/\mathbb{Q})$ are all $1$-dimensional, and therefore by abelian reciprocity the functions appearing in the Artin factorization of $\zeta_K(s)$ are moreover distinct primitive Dirichlet $L$-functions, among them $\zeta(s)$. Label the functions $L_1(s),\dots,L_n(s)$, where $L_1(s) = \zeta(s)$, say, and $n = [K:\mathbb{Q}]$. We assume the generalized Riemann hypothesis for the Dirichlet $L$-functions, simplicity of all nontrivial zeros (a priori none at $s = \frac{1}{2}$), and that no two Dirichlet $L$-functions share a nontrivial zero. Let $k \in \mathbb{R}$. Then,
\begin{equation*}
    \sum_{\substack{0 < \gamma \leq T \\ \zeta_K(\frac{1}{2} + i\gamma) = 0}} |\zeta_K'(\rho)|^{2k} = \sum_{i=1}^n \sum_{\substack{0 < \gamma_i \leq T \\ L_i(\frac{1}{2} + i\gamma_i) = 0}} |L_i'(\rho)|^{2k} \,\prod_{j \neq i} |L_j(\rho)|^{2k},
\end{equation*}
where we maintain the convention $\rho = \frac{1}{2} + i\gamma_i$. 
Upon normalizing the $i$th summand by the factor $N_i(T) = \#\{0 \leq \gamma_i \leq T \mid L_i(\frac{1}{2} + i \gamma_i) = 0 \}$, we could anticipate that
\begin{align*}
   & \frac{1}{N_i(T)} \sum_{\substack{0 < \gamma_i \leq T \\ L_i(\frac{1}{2} + i\gamma_i) = 0}} |L_i'(\rho)|^{2k} \,\prod_{j \neq i} |L_j(\rho)|^{2k} \\ &\phantom{\,\,\,\,}\asymp \left( \frac{1}{N_i(T)} \sum_{\substack{0 < \gamma_i \leq T \\ L_i(\frac{1}{2} + i \gamma_i) = 0}} |L_i'(\rho)|^{2k} \right) \prod_{j \neq i} \left(\frac{1}{ N_i(T)} \sum_{\substack{0 < \gamma_i \leq T \\ L_i(\frac{1}{2} + i\gamma_i) = 0}} |L_j(\rho)|^{2k} \right).
\end{align*}
Noting its similarity to the sum $J_{k}(T)$, and recalling the known asymptotic $N_i(T) \asymp T(\log T)$ (see~\cite[Theorem 14.5]{MV}), we may posit that
\begin{equation*}
    \frac{1}{N_i(T)} \sum_{0 < \gamma_i \leq T} |L_i'(\rho)|^{2k}~\asymp (\log T)^{k(k+2)},
\end{equation*}
as in the case of the Riemann zeta function (cf.~\cite{Gonek1}). Meanwhile, to address the inner sums
\begin{equation*}
    \frac{1}{N_i(T)} \sum_{\substack{0 < \gamma_i \leq T \\ L_i(\frac{1}{2} + i\gamma_i) = 0}} |L_j(\rho)|^{2k},
\end{equation*}
where $j \neq i$, it would be most favorable to impose some restriction on the proximity of the zeros $\gamma_i$ of $L_i$ to lower values of the $L_j$; this is because we are taking $k = -1$. In this case, we expect that these normalized sums are $O_{\varepsilon}(T^\varepsilon)$ for all $\varepsilon > 0$. This would suggest that $J_{-1}^K(T) \ll_{\alpha} T^{1 + \alpha}$ for all $\alpha > 0$, which would be sufficient for \cref{thm:limiting-distribution-abelian}. In fact, we believe that $J_{-1}^K(T) \ll_K T$ holds for all abelian number fields $K$.

%% file: 5details.tex
In this section, we give proofs of some of the technical statements that we have used in prior sections. 

\subsection{Bounds in horizontal strips}

In this section, we outline a proof of \cref{lem:large-horiz-strip}. This proof is heavily based on the proof of \cite[Theorem 13.22]{MV}, the corresponding statement for $K=\mathbb Q$. We will run through the intermediate statements needed for this result as it is proven in \cite{MV}, and in each step explain the necessary modifications in order to generalize to the setting of Dedekind zeta functions.

\begin{lem}[cf.~\protect{\cite[Corollary~10.5]{MV}}]\label{lem:MV-10.5}
Suppose $-1 \leq \sigma \leq \frac{1}{2}$. Then
$$\left \vert \frac{\zeta_K(\sigma + iT)}{\zeta_K(1 - \sigma - iT)} \right \vert \asymp \vert T \vert^{(1/2-\sigma)n_K^2}.$$
\end{lem}

This follows from the reflection formula for Dedekind zeta functions and Stirling's approximation. By the Hadamard product for the completed Dedekind zeta function, one gets the following formula for the logarithmic derivative of $\zeta_K$.

\begin{lem}[cf.~\protect{\cite[Lemma~12.1]{MV}}]\label{lem:MV-12.1}
We have
$$\frac{\zeta_K'(s)}{\zeta_K(s)} = -\left(\frac{1}{s} + \frac{1}{s-1}\right)+ \sum_{\vert \gamma - t \vert \leq 1}\left(\frac{1}{\rho} + \frac{1}{s - \rho}\right) + O(\log\tau)$$
uniformly for $-1 \leq \sigma \leq 2$, where $\tau = |t| + 4$.
\end{lem}

Define $\Lambda^K(n)$ to be, as in the proof of \cref{lem:high-deriv-log-deriv}, such that
$$\frac{\zeta_K'(s)}{\zeta_K(s)}=\sum_{n=1}^\infty \Lambda^K(n)n^{-s}.$$
We will use the bound
\begin{equation}\label{eq:lambda-K-bound}
0\leq \Lambda^K(n)\leq n_K\Lambda(n)
\end{equation}
for all positive integers $n$ in many of the following results. It may be proven by a simple consideration of the powers of prime ideals of $\mc O_K$ with given norm; for the upper bound, equality is reached for $n$ a power of a prime $p$ which splits completely in $K$. The first result needed for \cref{lem:large-horiz-strip} is the following.

\begin{lem}[cf.~\protect{\cite[Theorem~13.13]{MV}}]\label{lem:MV-13.13} Assume the Riemann hypothesis for $K$. Then
$$\left \vert \frac{\zeta_K'(s)}{\zeta_K(s)}\right \vert \leq \sum_{n \leq (\log \tau)^2} \frac{\Lambda^K(n)}{n^\sigma} + O((\log \tau)^{2 - 2\sigma})$$
uniformly for $\frac{1}{2} + \frac{1}{\log \log \tau} \leq \sigma \leq \frac{3}{2}$, $\vert t \vert ~\geq 1$.
\end{lem}

This lemma is proven essentially identically to the corresponding result from \cite{MV}, and a proof along very similar lines is given in this paper for \cref{lem:high-deriv-log-deriv}. In particular, the only modification to the proof in \cite{MV} that is needed results from the location and multiplicity of trivial zeros (see \eqref{eq:trivial-zero-sum}). We remark that, in light of \eqref{eq:lambda-K-bound}, this lemma implies the weaker bound
\begin{equation}\label{eq:weak-MV-13.13}
\left \vert \frac{\zeta_K'(s)}{\zeta_K(s)}\right \vert \leq n_K\sum_{n \leq (\log \tau)^2} \frac{\Lambda(n)}{n^\sigma} + O((\log \tau)^{2 - 2\sigma}),
\end{equation}
in which a multiplicative constant is sacrificed for conceptual simplicity.

\begin{lem}[cf.~\protect{\cite[Corollary~13.16]{MV}}]\label{cor:MV-13.16} Assume the Riemann hypothesis for $\zeta_K(s)$. Then, for $|t|~\geq 1$,
$$
|\log\zeta_K(s)|~\leq n_K\cdot\begin{cases}
    \log\frac1{\sigma-1}+O(\sigma-1)
    &\text{if }
    1+\frac1{\log\log\tau}\leq\sigma\leq\frac32\\
    \log\log\log\tau+O(1)
    &\text{if }
    1-\frac1{\log\log\tau}\leq\sigma\leq1+\frac1{\log\log\tau}\\
    \log\frac1{1-\sigma}+O\left(\frac{(\log \tau)^{2-2\sigma}}{(1-\sigma)\log\log\tau}\right)
    &\text{if }
    \frac12+\frac1{\log\log\tau}\leq\sigma\leq1-\frac1{\log\log\tau}.
    \end{cases}
$$
\end{lem}

The proof of this corollary is identical to the proof of \cite[Corollary 13.16]{MV}, with \eqref{eq:weak-MV-13.13} used in the place of \cite[Theorem 13.13]{MV}. We also need two additional results about the zeros of $\zeta_K(s)$, direct analogues of \cite[Lemma 13.19]{MV} and \cite[Lemma 13.20]{MV}, with analogous proofs. Firstly, the bound in \cref{lem:MV-13.13} can be used to show that
\begin{equation}\label{eq:MV-13.19}
N_K\left(T+\frac1{\log\log T}\right)-N_K(T)\ll \frac{\log T}{\log\log T},
\end{equation}
a refinement of \cref{cor:zero-count}. Secondly, an application of \cref{lem:MV-12.1}, combined with the counting result from the preceding equation, shows that, for $|t|~\geq 1$ and $|\sigma-\frac{1}{2}|~\leq \frac{1}{\log \log \tau}$, then
\begin{equation}\label{eq:MV-13.20}
\frac{\zeta_K'(s)}{\zeta_K(s)}=\sum_{|\gamma-t|\leq 1/\log\log\tau}\frac1{s-\rho}+O(\log\tau).
\end{equation}
These results are enough to show \cref{lem:large-horiz-strip} following the proof of \cite[Theorem 13.22]{MV}. 

\begin{proof}[Proof of \cref{lem:large-horiz-strip}]

\cref{lem:MV-10.5} implies that we need only to consider $\frac{1}{2}<\sigma<2$. From here, \cref{cor:MV-13.16} shows the result for $\sigma>\frac{1}{2}+\frac{1}{\log\log T}$, and integrating \eqref{eq:MV-13.20} over a short horizontal segment gives a bound of the desired form, but in terms of the distances between $T$ and various ordinates of the zeros of $\zeta_K(s)$. Integrating this bound over $n\leq T<n+1$ for any $n$ gives that the result holds on average, and thus for some fixed $T=T_n$, as desired.
\end{proof}

\subsection{Bounds on residues at the trivial zeros}

We now prove \cref{lem:triv-zero-bound}, following the method of the proof of \cref{lem:residue-at-zero}.

\begin{proof}[Proof of \cref{lem:triv-zero-bound}] First, the residue at $0$ can be written as
$$\lim_{s\to 0}\frac{d^{r_1+r_2-1}}{ds^{r_1+r_2-1}}\frac{s^{r_1+r_2-1}x^s}{\zeta_K(s)}=\lim_{s\to 0}\frac{d^{r_1+r_2-1}}{ds^{r_1+r_2-1}}x^sf(s)$$
for some $f$ independent of $x$ with neither a pole nor a zero at $s=0$. Evaluating the derivative explicitly yields that this residue is polylogarithmic in $x$, with exponent $r_1 + r_2 - 1$. 

For $k>0$, by \cref{thm:fe}, we write
\begin{equation}\label{eqn:expandedfe}
    \frac{x^s}{s\zeta_K(s)} = \frac{2^{r_2}\pi^{n_K}}{D_K^{1/2}} \left(\frac{x D_K}{(2\pi)^{n_K}}\right)^s\left( \frac{1}{s\zeta_K(1-s)} \right) \left( \frac{1}{\sin \frac{\pi s}2} \right)^{r_1} \left( \frac{1}{\sin \pi s} \right)^{r_2} \left( \frac{1}{\Gamma(1-s)} \right)^{n_K},
\end{equation}
where we have used the reflection formula $\Gamma(s)\Gamma(1-s) = \pi / \sin \pi s$. At $s = -k$, none of these terms are zero. The fourth term has a pole of order $r_1$ if $k$ is even, and the fifth term has a pole of order $r_2$ regardless. Define $\epsilon$ to be $1$ if $k$ is even and $0$ otherwise, so that $\frac{x^s}{s\zeta_K(s)}$ has a pole of order $r_1\epsilon+r_2$ at $s=-k$. Define
\begin{align*}
f(s) &=\frac{(s+k)^{r_1\epsilon+r_2}x^s}{s\zeta_K(s)}\\
g_1(s)&=s\log\left(\frac{x\adisc_K}{(2\pi)^{n_K}}\right)\\
g_2(s)&=\log(s\zeta_K(1-s))\\
g_3(s)&=\epsilon\log(s+k)-\log\sin\frac{\pi s}2\\
g_4(s)&=\log(s+k)-\log\sin\pi s\\
g_5(s)&=\log\Gamma(1-s)
\end{align*}
so that
$$g(s)\coloneqq\log f(s)=\log\left(\frac{2^{r_2}\pi^{n_K}}{\adisc_K^{1/2}}\right)+g_1(s)-g_2(s)+r_1g_3(s)+r_2g_4(s)-n_Kg_5(s),$$
and so that
$$\res_{s=-k}\frac{x^s}{s\zeta_K(s)}=\lim_{s\to -k}\frac{d^{r_1\epsilon+r_2-1}}{d^{r_1\epsilon+r_2-1}s}f(s).$$
We may show that, for $\ell\leq n_K$,
\begin{align*}
\left|g_1^{(\ell)}(k)\right|&\leq \log\left(\frac{xD_K}{(2\pi)^{n_K}}\right)\\
\left|g_2^{(\ell)}(k)\right|,\left|g_3^{(\ell)}(k)\right|,\left|g_4^{(\ell)}(k)\right|&\ll_{r_1,r_2} 1\\
\left|g_5^{(\ell)}(k)\right|&\ll \log k.
\end{align*}
This shows that $g^{(\ell)}(s)\ll_{r_1,r_2} \log(kx\adisc_K)$. In addition,
\begin{align*}
\left|\lim_{s\to -k}f(s)\right|
&=\left|\left(\frac{2^{r_2}\pi^{n_K}}{\adisc_K^{1/2}}\right)
\frac{\left(\frac{x\adisc_K}{(2\pi)^{n_K}}\right)^{-k}}{\lim_{s\to -k}s\zeta_K(1-s)}\left(\frac 2\pi\right)^{\epsilon r_1}\left(\frac 1\pi\right)^{r_2}\frac1{k!^{n_K}}\right|\\
&\ll_{r_1,r_2}\frac{(2\pi)^{kn_K}}{x^k\adisc_K^{k+1/2}k!^{n_K}}\cdot \frac{1}{\left|\lim_{s\to -k}s\zeta_K(1-s)\right|}\ll \frac{(2\pi)^{kn_K}}{kx^k\adisc_K^{k+1/2}k!^{n_K}}.
\end{align*}
As in the proof of \cref{lem:residue-at-zero}, since $f'(s)=g'(s)f(s)$, we have that $f^{(j)}(s)/f(s)$ is a polynomial of degree at most $j$ in the first $j$ derivatives of $g$. This finishes the proof.
\end{proof}

\subsection{Proof of the explicit formula}
To prove \cref{lem:explicit-formula}, the explicit formula for $M_K(x)$, we need a rather precise truncated form of Perron's formula. This can be found by combining the statements of \cite[Lemma 3.12]{Titchmarsh} and \cite[Lemma 2]{Ng}, and a similar result can be found in \cite[Corollary 5.3]{MV}.

\begin{lem}\label{lem:perron} Suppose $f(s)=\sum_{n=1}^\infty a_nn^{-s}$ is absolutely convergent for $\Real(s)>1$, and let $\Phi(n)$ be a positive and non-decreasing function such that $a_n\ll\Phi(n)$. Assume that
$$\sum_{n=1}^\infty \frac{|a_n|}{n^\sigma}=O\left(\frac1{(\sigma-1)^\alpha}\right)$$
as $\sigma$ tends to $1$ from above. Pick $w=u+iv$ and $c>0$ such that $u+c>1$, and pick $T>0$. Then, for $x>0$,
\begin{dmath*} 
\sum_{n \leq x}\frac{a_n}{n^w}=\frac1{2\pi i}\int_{c-iT}^{c+iT}f(w+s)\frac{x^s}sds+O\left(\frac{x^c}{T(u+c-1)^\alpha}+\frac{\Phi(2x)x^{1-u}\log x}{T}+E_1(x,T)\right),
\end{dmath*}
where
$$E_1(x,T)\ll \Phi(2x)x^{-u}$$
if $x > 1$, and, if $N$ is the nearest integer to $x$ (besides possibly $x$ itself),
$$E_1(x,T)\ll \frac{\Phi(N)x^{1-u}}{T|x-N|}$$
for any $x>0$.
\end{lem}

We now give a bound on the integral found in our application of Perron's formula. This is done separately from the main proof of the explicit formula in order to elucidate the locations where the appropriate error terms arise. 

\begin{lem}\label{lem:integral-bound-for-ef} Fix a number field $K$, and let $\mathcal T$ be as in \cref{lem:large-horiz-strip}. For any $x>0$, $0<T\in\mathcal T$, and $1<c \leq 2$, write
\begin{dmath*}
\frac1{2\pi i}\int_{c-iT}^{c+iT}\frac{x^s}{s\zeta_K(s)}ds= \sum_{|\gamma| < T} \frac{x^{\rho}}{\rho \zeta_K'(\rho)} + \sum_{k=0}^{\infty}\res_{s=-k}\frac{x^s}{s\zeta_K(s)}+E_2(x,c,T).
\end{dmath*}
Then, for any $\epsilon>0$, $E_2(x,c,T)=O_x(1/T^{1-\epsilon})$. When $x>1$, $E_2(x,c,T)=O(x^c/T^{1-\epsilon})$ uniformly.
\end{lem}
\begin{proof} Pick a sufficiently large positive half-integer $U$ (we will let $U\to \infty$), and integrate around the rectangular contour with vertices $c\pm iT$ and $-U\pm iT$. This gives
\begin{dmath}\label{eq:split-contour}
\frac1{2\pi i}\int_{c-iT}^{c+iT}\frac{x^s}{s\zeta_K(s)}ds=\sum_{|\gamma|<T}\frac{x^\rho}{\rho\zeta_K'(\rho)}+\sum_{k=0}^{\lfloor U\rfloor} \res_{s=-k}\frac{x^s}{s\zeta_K(s)}+\frac1{2\pi i}\left(\int_{c-iT}^{-U-iT}+\int_{-U-iT}^{-U+iT}+\int_{-U+iT}^{c+iT}\right)\frac{x^s}{s\zeta_K(s)}ds.
\end{dmath}
We first bound these integrals using the functional equation for $\zeta_K(s)$ (\cref{thm:fe}). First, we recall the identities and bounds
\begin{align*}
\Gamma(s)\Gamma(1-s) &= \frac{\pi}{\sin \pi s}\\
\Gamma(s)^{-1}&\ll e^{\sigma-(\sigma-1/2)\log\sigma+\pi|t|/2}\\
|\sin\pi s|^{-1}&\ll e^{-\pi |t|}\\
\left|\sin\frac{\pi s}2\right|^{\pm 1}&\ll e^{\pm\pi |t|/2}\\
\zeta_K(s)^{-1}&\ll 1\text{ for }\sigma>2,
\end{align*}
which hold for $s$ bounded away from integers. As the paths of the three integrals in \eqref{eq:split-contour} always stay a distance of $\frac{1}{4}$ away from each integer, we have for $s$ on such a path, and with real part not between $-1$ and $2$, that
\begin{equation}\label{eq:integrand-bound}
\frac{x^s}{s\zeta_K(s)}\ll_K \frac{\left(x\adisc_K/(2\pi)^{n_K}\right)^\sigma}{T}e^{n_K\left[1-\sigma-(1/2-\sigma)\log(1-\sigma)\right]}.
\end{equation}
Thus, the middle integral is asymptotically at most $x^Ue^{-U\log U+O(U)}$ uniformly in $x$ and $T$ as $U\to\infty$. For the two horizontal integrals, we first bound the portion where $\Real(s)$ runs from $-1$ to $-U$. Here, we may use \eqref{eq:integrand-bound}. Since the integral in the right-hand side of
$$\int_{-1 + iT}^{-U + iT}\frac{x^s}{s\zeta_K(s)}\,ds \ll \frac1T\int_{-1}^{-U}\left(\frac{x\adisc_K}{(2\pi)^{n_K}}\right)^\sigma e^{n_K\left[1-\sigma-(1/2-\sigma)\log(1-\sigma)\right]}\,d\sigma$$
converges as $U\to\infty$, this is $O(1/T)$ for fixed $x$, and for $x>1$ the dependence is $O(x^{-1})$. As a result, we may write
\begin{dmath*}
\frac1{2\pi i}\int_{c-iT}^{c+iT}\frac{x^s}{s\zeta_K(s)}ds=\sum_{|\gamma|<T}\frac{x^\rho}{\rho\zeta_K'(\rho)}+\sum_{k=0}^\infty \res_{s=-k}\frac{x^s}{s\zeta_K(s)}+\frac1{2\pi i}\left(\int_{c-iT}^{-1-iT}+\int_{-1+iT}^{c+iT}\right)\frac{x^s}{s\zeta_K(s)}ds+E_3(x,T),
\end{dmath*}
where $E_3(x,T)=O_x(1/T)$ for $x > 0$, and $E_3(x,T)=O(1/xT)$ for $x>1$. Finally, for the remaining integrals, we use the definition of $\mathcal T$. Since $|\zeta_K(\sigma\pm iT)|^{-1}\ll T^\epsilon$ uniformly in $-1\leq \sigma \leq 2$ for $T\in\mathcal T$ and any $\epsilon>0$, we have
$$\frac{x^s}{s\zeta_K(s)}\ll \frac{\max(x^c,1/x)}{T^{1-\epsilon}}$$
for any $s$ on one of our remaining line segments. This bounds our integrals by $\max(x^c,1/x)/T^{1-\epsilon}$, which is enough.
\end{proof}

We are now ready to show the explicit formula for $M_K(x)$.

\begin{proof}[Proof of \cref{lem:explicit-formula}] First, note that $|\mu_K(n)|$ is bounded by the coefficient of $n^{-s}$ in the Dirichlet series expansion of $\zeta_K(s)$. As a result,
$$\sum_{n=1}^\infty \frac{|\mu_K(n)|}{n^\sigma}=O\left(\frac{1}{\sigma - 1}\right) \text{ as } \sigma \rightarrow 1^+,$$
since the pole of $\zeta_K(s)$ at $s = 1$ is simple. Define
$$\Phi(x)=x^{\frac{C}{\log\log(x+10)}}$$
where $C$ is as in \cref{lem:dedekind-coeff-bound}, so that $\Phi(x)$ is positive, non-decreasing, and satisfies $\mu_K(n)\ll \Phi(n)$. Now, apply \ref{lem:perron} to $f(s)=\frac1{\zeta_K(s)}$, $w=0$, $\alpha=1$, $T>0$, and with $1<c\leq 2$ to be chosen later. This gives
$$M_K(x)=\frac1{2\pi i}\int_{c-iT}^{c+iT}\frac{x^s}{s\zeta_K(s)}ds+O\left(\frac{x^c}{T(c-1)}+\frac{\Phi(2x)x\log x}{T}+E_1(x,T)\right),$$
where $E_1(x,T)$ is as in the statement of \cref{lem:perron}. By \cref{lem:integral-bound-for-ef}, we have
\begin{dmath*}
M_K(x)={\sum_{|\gamma|\leq T} \frac{x^\rho}{\rho\zeta_K'(\rho)}+\sum_{k=0}^{\infty}\res_{s=-k}\frac{x^s}{s\zeta_K(s)}}+O\left(\frac{x^c}{T(c-1)}+\frac{\Phi(2x)x\log x}{T}+E_1(x,T)+E_2(x,c,T)\right).
\end{dmath*}
This is enough to show our inequality on the error term when $x > 0$, as each term in the above is asymptotically bounded by $T^{-(1-\epsilon)}$ (any $1<c\leq2$ will suffice). When $x > 1$, we have
$$E(x,T)\ll \frac{x^c}{T(c-1)}+\frac{\Phi(2x)x\log x}{T}+\frac{x^c}{T^{1-\epsilon}}+\Phi(2x).$$
Setting $c=1+(\log x)^{-1}$ gives
$$E(x,T)\ll\frac{x^{1+C/\log \log x}\log x}T+\frac{x}{T^{1-\epsilon}}+x^{C/\log\log x}.$$

Finally, \cref{lem:triv-zero-bound} implies that the sum of the residues at negative integers is convergent as a series in $k$, and its dependence on $x$ is $O((\log x)^{r_1 + r_2})$ depending on $K$.
\end{proof}

%% file: 6one-sided-growth.tex
We have mostly avoided the discussion of the fields $K$ for which $\zeta_K(s)$ has nontrivial non-simple zeros. Although these fields are not problematic if the goal is to disprove the \mertensconjecture over $K$, they pose a distinct challenge when trying to study properties like limiting distributions of normalizations of $M_K(x)$. In this setting, stronger hypotheses are needed. In this section, we posit and motivate a particularly strange property that Mertens functions of non-abelian number fields may satisfy.

\begin{conj}\label{conj:sign-change}
There are infinitely many number fields $K$ for which the Mertens function $M_K(x)$ changes sign only finitely many times. 
\end{conj}

The heuristics behind this conjecture are covered in \cref{subsec:conj-heuristics}.


\subsection{Complexities of non-abelian number fields}
\label{subsec:nonabelian-complexities}

A main tool in studying the limiting distribution of $M_K(x)/\sqrt{x}$ is the explicit formula (\cref{lem:explicit-formula}). This formula was derived via Perron's formula, which allows one to express $M_K(x)$ as a sum of residues of $\frac{x^s}{s\zeta_K(s)}$. If $\rho$ is a simple zero of $\zeta_K(s)$, then this residue has a particularly simple formula and is of order $x^{1/2}$ as a function of $x$. However, if $\ord_{s = \rho}\zeta_K(s) = n > 1$, then this residue is a linear combination of terms $x^{1/2}(\log x)^k$ where $k < n$. In particular, the coefficients of these terms involve large positive and negative powers of higher derivatives of $\zeta_K(x)$, quantities whose moments are not well understood. 

Assuming some standard conjectures, the zeros of $\zeta_K(s)$ for a non-abelian number field $K$ (with the possible exception of a zero at $s = \frac{1}{2}$) have multiplicity at most the largest dimension of an irreducible complex representation of $\gal(K/\QQ)$. This motivates the following question.

\begin{quest}
If the largest irreducible representation of $\gal(K/\QQ)$ has dimension $n$, then does $M_K(x)/(\sqrt x\log^{n-1}(x))$ have a limiting distribution?
\end{quest}

Moreover, assume that the $\zeta_K(s)$ has a zero at $s = \frac{1}{2}$ of multiplicity larger than that~of~any other nontrivial zero. As described in \cref{prop:noRH-rightmost-real}, this term does not oscillate between large positive and negative values (like terms associated with non-real zeros), and will dominate all others in the explicit formula for $M_K(x)$. Hence, one may ask the following question.

\begin{quest}
If $\ord_{s = \frac{1}{2}}\zeta_K(s) > \ord_{s = \rho} \zeta_K(s)$ for all $\rho \neq \frac{1}{2}$, does $M_K(x)$ change sign finitely many times?
\end{quest}

One may ask if such number fields satisfying the hypothesis of this question even exist. This is answered in the affirmative, assuming appropriate standard conjectures, by the work of Louboutin~\cite{Louboutin1}.

It now remains to identify number fields $K$ of interest, as well as give heuristic bounds on certain sums over zeros, which we turn to now. In what follows, it will be relevant to state the following assumptions on Artin $L$-functions, for a Galois extension of number fields $L/K$:

\begin{conj}[Artin holomorphy conjecture (AHC)]\label{conj:artin-hol} For any character $\chi$ of a nontrivial irreducible representation of $\gal(L/K)$, the function $L(s,\chi,L/K)$ is holomorphic.
\end{conj}


\begin{conj}[Simplicity hypothesis ($\text{SH}_{K/\QQ}$)]\label{conj:artin-sh} If $\chi$ is an irreducible character of $\gal(K/\QQ)$, then the nontrivial zeros of $L(s,\chi,K/\QQ)$ are simple. 
\end{conj}

\begin{conj}[Independence conjecture ($\text{IC}_{K/\QQ}$)]\label{conj:artin-indep} If $\chi_1$, $\chi_2$ are distinct irreducible characters of $\gal(K/\QQ)$, then $L(s,\chi_1,K/\QQ)$ and $L(s,\chi_2,K/\QQ)$ share no nontrivial zeros, \emph{except possibly for a zero at $s=\frac{1}{2}$}.
\end{conj}

\subsection{Candidate number fields}

The first appearance of a number field $K$ satisfying $\zeta_K(\frac{1}{2}) = 0$ occurred in a paper of Armitage~\cite{Armitage}, who studied a degree $48$ extension of $\mathbb{Q}$ which had first appeared in a work of Serre~\cite{Serre}. Later, Serre (unpublished, see~\cite[Chapter~2]{NgThesis}) discovered the simpler example
$$K=\QQ\left(\sqrt{(5+\sqrt 5)(41+\sqrt{205})}\right),$$
which is Galois with quaternionic Galois group $Q_8$. 

In emulation of this latter construction, we identify as candidates for $\gal(K/\mathbb{Q})$ the following well-known class of groups to which $Q_8$ belongs. These groups possess convenient representation-theoretic properties, which allow control over the behavior of $\zeta_K(s)$ at $s=\frac{1}{2}$.

\begin{defn}
Let $n \geq 2$ an integer. We define the \textit{dicyclic group $Q_{4n}$} by the presentation
\begin{equation*}
    Q_{4n} = \langle a, b: a^{2p} = 1, a^{p} = b^{2}, b^{-1} ab = a^{-1} \rangle.
\end{equation*}
This is a solvable non-abelian group of order $4n$.
\end{defn}

In what follows, let $p \geq 3$ be an odd prime, and let $N/\mathbb{Q}$ be a normal extension with Galois group $G = Q_{4p}$, called a \textit{dicyclic number field} of degree $4p$. We recall the Galois theory of such fields, which in be found in the introduction to \cite{Louboutin1}.

If $p$ is not totally ramified in $N/\mathbb{Q}$, then by \cite[Theorem~1]{Louboutin1}, the root number~$W_{N/\mathbb{Q}}(\psi) = W_{N/L}(\chi)$ does not depend on $\chi$ but~just on the field $N$, so either there are $(p-1)/2$ two-dimensional characters with root number $-1$ or there are none. Hence, assuming \hyperref[conj:artin-sh]{$\text{SH}_{N/\QQ}$}, the order of the zero at $s = \frac{1}{2}$ of the Dedekind zeta function is either $p-1$ or~zero. Moreover, there exist infinitely many examples of dicyclic fields $N$ which produce either case. If $p$ is totally ramified in $N/\mathbb{Q}$, which can only occur when $p \equiv 1 \pmod{4}$, and if $L = \mathbb{Q}(\sqrt{p})$, then half of the $(p-1)/2$ quaternionic characters have root number $-1$ and half have root number $+1$. Hence, assuming \hyperref[conj:artin-sh]{$\text{SH}_{N/\QQ}$}, the order of the zero at $s = \frac{1}{2}$ is $(p-1)/2$. 

Now that we have built up the multiplicity of the zero at $s = \frac{1}{2}$, we invoke the conditions \hyperref[conj:artin-sh]{$\text{SH}_{N/\QQ}$} and \hyperref[conj:artin-indep]{$\text{IC}_{N/\QQ}$} in order to tame the multiplicities of all the other nontrivial zeros. Equipped with these assumptions, the largest dimension of an irreducible representation of a dicyclic number field would be $2$, so that all other zeroes besides $s = \frac{1}{2}$ should have order at most~$2$.

\begin{thm}\label{thm:dicyclic-nf-findings}
Let $p \geq 3$ be a prime and let $N$ be a dicyclic number field of order $4p$. Assume \textup{\hyperref[conj:artin-sh]{$\text{SH}_{N/\QQ}$}}.
\begin{enumerate}[label=(\alph*),font=\normalfont]
    \item If $p$ is not totally ramified in $N/\QQ$, and $W(\chi) = -1$ for all characters $\chi$ on $H$ of order $2p$, then the order of the zero of $\zeta_N(s)$ at $s = \frac{1}{2}$ is $p - 1$. 
    \item If $p$ is totally ramified in $N/\QQ$, and $L = \QQ(\sqrt{p})$, then the order of the zero of $\zeta_N(s)$ at $s = \frac{1}{2}$ is $(p-1)/2$.
\end{enumerate}
Assuming moreover \textup{\hyperref[conj:artin-indep]{$\text{IC}_{N/\QQ}$}}, and for $p$ sufficiently large, this order is higher than that of any other nontrivial zero. 
\end{thm}

Moreover, there are infinitely many examples of $N$ as in case (a) assuming \hyperref[conj:artin-sh]{$\text{SH}_{N/\QQ}$}.

\subsection{Heuristics in support of one-sided growth of $M_K(x)$}
\label{subsec:conj-heuristics}

From now on, we shall specialize to the scenario of \cref{thm:dicyclic-nf-findings}, where $K \coloneqq N$. Hence, $K$ is a dicyclic number field satisfying \hyperref[conj:artin-sh]{$\text{SH}_{K/\QQ}$} and \hyperref[conj:artin-indep]{$\text{IC}_{K/\QQ}$}, with the function $\zeta_K(s)$ having a zero of multiplicity $\mu \geq 3$ at $s = 1/2$ and all other nontrivial zeros of multiplicity equal to either $1$ (simple) or $2$ (double). In the notation of \cref{lem:explicit-formula}, for all $T_n \in \mathcal{T}$ and $x > 1$ we have, under the Riemann hypothesis for $\zeta_K(s)$,
\begin{equation}\label{eq:dicyclic-explicit-formula}
    M_K(x) = \res_{s = 1/2} \frac{x^s}{s\zeta_K(s)} + \sum_{\substack{|\gamma| \leq T_n \\\rho \text{ simple}}} \res_{s = \rho} \frac{x^s}{s\zeta_K(s)} +\sum_{\substack{|\gamma| \leq T_n \\\rho \text{ double}}} \res_{s = \rho} \frac{x^s}{s\zeta_K(s)} + E(x,T_n),
\end{equation}
As noted already, the leading term above grows one-sidedly in $x$ as $x^{1/2}(\log x)^{\mu - 1}$, so we seek to restrain the growth of the remaining terms. We have
\begin{align*}
    \res_{s = \rho} \frac{x^s}{s\zeta_K(s)}  = \frac{x^{\rho}}{\rho \zeta_K'(\rho)} \text{ and } \res_{s = \rho} \frac{x^s}{s\zeta_K(s)} = \frac{2x^{\rho}}{\rho \zeta_K''(\rho)} \left( \log x - \frac{1}{\rho} - \frac{\zeta_K'''(\rho)}{\zeta_K''(\rho)}\right),
\end{align*}
when $s = \rho$ is a simple and double zero, respectively. In the former case, one suitable bound would be a generalized form of \eqref{eq:gonek-hejhal-nfield} which is applicable to Dedekind zeta functions having nontrivial zeros of multiplicity, but which targets only the simple ones, namely
\begin{equation*}
    \sum_{\substack{0 < \gamma \leq T \\ \rho \text{ simple}}} \frac{1}{|\zeta_K'(\rho)|^2} \ll T,
\end{equation*}
give or take the addition of a low power of $\log T$ to the right-hand side. In this case, we may replace the $T_n$ in the second term of \eqref{eq:dicyclic-explicit-formula} with any $T > 0$, upon inserting an additional error term of
\begin{equation*}
    \left( \frac{x \log T}{T} \right)^{1/2}.
\end{equation*}
A similar estimate pertaining to the double zeros,
\begin{equation*}
    \sum_{\substack{0 < \gamma \leq T \\ \rho \text{ double}}} \frac{1}{|\zeta_K''(\rho)|^2} \ll T,
\end{equation*}
and an estimate of the form
\begin{equation*}
    \sum_{\substack{0 < \gamma \leq T \\ \rho \text{ double}}} \frac{|\zeta_K'''(\rho)|}{|\rho||\zeta_K''(\rho)|^2} \ll (\log T)^a,
\end{equation*}
with $a$ a reasonably small integer, would also come in handy, once again allowing for the~replacement of $T_n$ with any $T > 0$ in \eqref{eq:dicyclic-explicit-formula} in exchange for residual error terms.  Finally, following the arguments in the proof of \cite[Theorem~1]{Ng} would produce the conclusion of \cref{conj:sign-change}, assuming that the resulting contribution of extra powers of $\log x$ has been smaller than $(\log x)^{\mu - 1}$.

%% file: 7conclusions-and-conjectures.tex
In this article, we have disproved the \mertensconjecture for most number fields, and have conditionally shown the existence of limiting distributions for abelian number fields. There are, however, a number of questions prompted by our results.

First, as discussed in \cref{sec:one-sided-growth}, it is possible that the Mertens function of certain non-abelian number fields may change sign only finitely many times -- a surprising departure from the behavior over $\mathbb{Q}$. Furthermore, the question of existence of limiting distributions for non-abelian number fields was also posed.

We have also shown in \cref{sec:disproof} that for a fixed signature, $M_K^+ > 1$ for all but finitely many number fields given conventional assumptions on the zero structure of $\zeta_K(s)$. This has been achieved by bounding from below the maxima of the function $x^{-1/2}(M_K(x) + M_K^*(x))$, which suffices for a disproof under the Riemann hypothesis and the simplicity of nontrivial zeros. As a companion to these findings, one may ask instead for upper bounds on the minima of this function, a question which pertains to $M_K^-$.


There are also several questions prompted by the existence of the limiting distributions $\nu_K$ (\cref{thm:limiting-distribution-abelian}). Although the Mertens conjecture has been known to be false, numerical evidence suggests that the bound $|M_\QQ(x)|~\leq \sqrt x$ holds most of the time (in the sense of logarithmic density). Namely, if $\nu_K$ is the logarithmic limiting distribution of $x^{-1/2}M_K(x)$~as in \cref{thm:limiting-distribution-abelian}, then preliminary computations suggest that $\nu_\QQ([-1, 1]) \geq 0.99999993366$ (cf.~\cite{Humphries14}). Nevertheless, it has not yet been proven that $\nu_\QQ([-1, 1]) > 0$. This~motivates the following two questions.

\begin{quest}
How does $\nu_K([-1, 1])$ vary with $K$?
\end{quest}

\begin{quest}
Let $\beta_K = \inf \{\alpha > 0 \mid \nu_K([-\alpha, \alpha]) \geq 0.5\}$. How does $\beta_K$ vary with $K$?
\end{quest}

In other words, is the statement $|M_K(x)|~\leq \sqrt{x}$ correct more or less frequently as the degree and discriminant of $K$ increase? Alternatively, one could study the asymptotic behavior of the distributions $\nu_K$ --- in particular, the behavior of its tails. In \cite[Section 4.3]{Ng}, Ng shows (assuming some conjectures on $J_{-1}(T)$ and $J_{-1/2}(T)$ as in \eqref{eq:gonek-hejhal}) that there are constants $c_1, c_2$ so that

$$
\exp\left(-\exp(c_1V^{\frac{4}{5}})\right) \ll \nu_\QQ([V, \infty)) \ll \exp\left(-\exp(c_2V^{\frac{4}{5}})\right)
$$

In particular, $\nu_\QQ([V, \infty)) > 0$ for all $V$, meaning $M_\QQ(x) > V\sqrt{x}$ a positive proportion of the time. That is, counterexamples to the Mertens conjecture over $\QQ$ have positive logarithmic density. One may ask how these asymptotics generalize to the settings of number fields. Thus, we pose the following two questions.

\begin{quest}
\label{quest:generaliz-ng-heuristic}
For an abelian number field $K$, are there $c_1, c_2$ so that

$$
\exp\left(-\exp(c_1V^{\frac{4}{5}})\right) \ll \nu_K([V, \infty)) \ll \exp\left(-\exp(c_2V^{\frac{4}{5}})\right)?
$$
\end{quest}

\begin{quest}
If \cref{quest:generaliz-ng-heuristic} is answered affirmatively, how should the constants vary with $K$?
\end{quest}

%% file: main.bbl
\newcommand{\etalchar}[1]{$^{#1}$}
\begin{thebibliography}{HKMS22}

\bibitem[And79]{Anderson}
R.~J. Anderson.
\newblock On the {M}ertens conjecture for cusp forms.
\newblock {\em Mathematika}, 26(2):236--249, 1979.

\bibitem[ANS14]{ANS}
A.~Akbary, N.~Ng, and M.~Shahabi.
\newblock {Limiting distributions of the classical error terms of prime number
  theory}.
\newblock {\em The Quarterly Journal of Mathematics}, 65(3):743--780, 2014.

\bibitem[Arm72]{Armitage}
J.~V. Armitage.
\newblock Zeta functions with a zero at {$s={1\over 2}$}.
\newblock {\em Inventiones mathematicae}, 15:199--205, 1972.

\bibitem[BBH{\etalchar{+}}71]{Bateman}
P.~T. Bateman, J.~W. Brown, R.~S. Hall, K.~E. Kloss, and R.~M. Stemmler.
\newblock Linear relations connecting the imaginary parts of the zeros of the
  zeta function.
\newblock In {\em Computers in number theory ({P}roc. {S}ci. {R}es. {C}ouncil
  {A}tlas {S}ympos. {N}o. 2, {O}xford, 1969)}, pages 11--19. Academic Press,
  London, 1971.

\bibitem[BGM15]{BGM}
H.~M. Bui, S.~M. Gonek, and M.~B. Milinovich.
\newblock A hybrid {E}uler--{H}adamard product and moments of {$\zeta'(\rho)$}.
\newblock {\em Forum Mathematicum}, 27(3):1799--1828, 2015.

\bibitem[BM92]{BM}
R.~Balasubramanian and V.~Kumar Murty.
\newblock Zeros of {D}irichlet {$L$}-functions.
\newblock {\em Annales Scientifiques de l'\'{E}cole Normale Sup\'{e}rieure.
  Quatri\`eme S\'{e}rie}, 25(5):567--615, 1992.

\bibitem[BT15]{BT}
D.~G. Best and T.~S. Trudgian.
\newblock Linear relations of zeroes of the zeta-function.
\newblock {\em Math. Comp.}, 84(294):2047--2058, 2015.

\bibitem[CN63]{Chandrasekharan}
K.~Chandrasekharan and R.~Narasimhan.
\newblock The approximate functional equation for a class of zeta-functions.
\newblock {\em Mathematische Annalen}, 152:30--64, 1963.

\bibitem[GHK07]{GHK}
S.~M. Gonek, C.~P. Hughes, and J.~P. Keating.
\newblock A hybrid {E}uler--{H}adamard product for the {R}iemann zeta function.
\newblock {\em Duke Mathematical Journal}, 136(3):507--549, 2007.

\bibitem[Gon89]{Gonek1}
S.~M. Gonek.
\newblock On negative moments of the {R}iemann zeta-function.
\newblock {\em Mathematika}, 36(1):71--88, 1989.

\bibitem[Gru82]{Grupp}
F.~Grupp.
\newblock On the {M}ertens conjecture for cusp forms.
\newblock {\em Mathematika}, 29(2):213--226, 1982.

\bibitem[Hea21]{Heap}
W.~Heap.
\newblock Moments of the {D}edekind zeta function and other non-primitive
  {$L$}-functions.
\newblock {\em Mathematical Proceedings of the Cambridge Philosophical
  Society}, 170(1):191--219, 2021.

\bibitem[Hej89]{Hejhal}
D.~A. Hejhal.
\newblock On the distribution of {$\log|\zeta'(\frac12+it)|$}.
\newblock In {\em Number theory, trace formulas and discrete groups ({O}slo,
  1987)}, pages 343--370. Academic Press, Boston, MA, 1989.

\bibitem[HKMS22]{hkms}
D.~Hu, I.~Kaneko, S.~Martin, and C.~Schildkraut.
\newblock Order of zeros of {D}edekind zeta functions.
\newblock {\em Proc. Amer. Math. Soc.}, 150(12):5111--5120, 2022.

\bibitem[HKO00]{HKO}
C.~P. Hughes, J.~P. Keating, and N.~O'Connell.
\newblock Random matrix theory and the derivative of the {R}iemann zeta
  function.
\newblock {\em Proceedings of the Royal Society of London. Proceedings. Series
  A. Mathematical, Physical and Engineering Sciences}, 456(2003):2611--2627,
  2000.

\bibitem[Hum13]{Humphries13}
P.~Humphries.
\newblock The distribution of weighted sums of the {L}iouville function and
  {P}\'{o}lya's conjecture.
\newblock {\em Journal of Number Theory}, 133(2):545--582, 2013.

\bibitem[Hum14]{Humphries14}
P.~Humphries.
\newblock On the {M}ertens conjecture for function fields.
\newblock {\em International Journal of Number Theory}, 10(2):341--361, 2014.

\bibitem[Hur18]{Hurst}
G.~Hurst.
\newblock Computations of the {M}ertens function and improved bounds on the
  {M}ertens conjecture.
\newblock {\em Mathematics of Computation}, 87(310):1013--1028, 2018.

\bibitem[Ing42]{Ingham}
A.~Ingham.
\newblock On two conjectures in the theory of numbers.
\newblock {\em American Journal of Mathematics}, 64:313--319, 1942.

\bibitem[IS99]{IS}
H.~Iwaniec and P.~Sarnak.
\newblock Dirichlet {$L$}-functions at the central point.
\newblock In {\em Number Theory in Progress, {V}ol. 2
  ({Z}akopane-{K}o\'{s}cielisko, 1997)}, pages 941--952. de Gruyter, Berlin,
  1999.

\bibitem[JP76]{JP}
W.~Jurkat and A.~Peyerimhoff.
\newblock A constructive approach to {K}ronecker approximations and its
  application to the {M}ertens conjecture.
\newblock {\em J. Reine Angew. Math.}, 286(287):322--340, 1976.

\bibitem[Jur73]{jurkat1973mertens}
W.~B. Jurkat.
\newblock On the {M}ertens conjecture and related general ${\Omega}$-theorems.
\newblock In H.~D. Diamond, editor, {\em Analytic Number Theory}, volume~24 of
  {\em Proceedings of Symposia in Pure Mathematics}, pages 147--158. American
  Mathematical Society, 1973.

\bibitem[KN12]{KN}
H.~Kadiri and N.~Ng.
\newblock Explicit zero density theorems for {D}edekind zeta functions.
\newblock {\em Journal of Number Theory}, 132(4):748--775, 2012.

\bibitem[KS00]{KS}
J.~P. Keating and N.~C. Snaith.
\newblock Random matrix theory and {$\zeta(1/2+it)$}.
\newblock {\em Communications in Mathematical Physics}, 214(1):57--89, 2000.

\bibitem[Lou00a]{LOUBOUTIN2000263}
S.~Louboutin.
\newblock Explicit bounds for residues of {D}edekind zeta functions, values of
  {$L$}-functions at {$s=1$}, and relative class numbers.
\newblock {\em Journal of Number Theory}, 85(2):263--282, 2000.

\bibitem[Lou00b]{Louboutin1}
S.~Louboutin.
\newblock Formulae for some {A}rtin root numbers.
\newblock {\em Tatra Mountains Mathematical Publications}, 20:19--29, 2000.
\newblock Number theory (Liptovsk\'{y} J\'{a}n, 1999).

\bibitem[Mer97]{Mertens}
F.~Mertens.
\newblock {\"{U}}ber eine zahlentheoretische {F}unktion.
\newblock {\em Sitzungsberichte der Kaiserlichen Akademie der Wissenschaften,
  Mathematisch-Naturwissenschaftliche Klasse, Abteilung 2a}, 106:761--830,
  1897.

\bibitem[MV07]{MV}
H.~L. Montgomery and R.~C. Vaughan.
\newblock {\em Multiplicative number theory. {I}. {C}lassical theory},
  volume~97 of {\em Cambridge Studies in Advanced Mathematics}.
\newblock Cambridge University Press, Cambridge, 2007.

\bibitem[Ng00]{NgThesis}
N.~Ng.
\newblock {\em Limiting distributions and zeros of {A}rtin {$L$}-functions}.
\newblock PhD thesis, University of British Columbia, 2000.

\bibitem[Ng04]{Ng}
N.~Ng.
\newblock The distribution of the summatory function of the {M}\"{o}bius
  function.
\newblock {\em Proceedings of the London Mathematical Society (3)},
  89(2):361--389, 2004.

\bibitem[OtR85]{OdlyzkoteRiele}
A.~M. Odlyzko and H.~J.~J. te~Riele.
\newblock Disproof of the {M}ertens conjecture.
\newblock {\em Journal f\"{u}r die reine und angewandte Mathematik},
  357:138--160, 1985.

\bibitem[RS94]{RS}
M.~Rubinstein and P.~Sarnak.
\newblock Chebyshev's bias.
\newblock {\em Experimental Mathematics}, 3(3):173--197, 1994.

\bibitem[Ser71]{Serre}
J.-P. Serre.
\newblock Conducteurs d'{A}rtin des caract\`eres r\'{e}els.
\newblock {\em Inventiones Mathematicae}, 14:173--183, 1971.

\bibitem[Tit86]{Titchmarsh}
E.~C. Titchmarsh.
\newblock {\em The theory of the {R}iemann zeta-function}.
\newblock The Clarendon Press, Oxford University Press, New York, second
  edition, 1986.
\newblock Edited and with a preface by D. R. Heath-Brown.

\end{thebibliography}
